\pdfoutput=1
\documentclass[reqno]{shinyart}
\usepackage[utf8]{inputenc}
\usepackage{shinybib}
\usepackage{enumitem}
\setlist[enumerate]{label=(\roman*)}
\usepackage{booktabs}
\usepackage[exponent-product=\cdot]{siunitx}
\usepackage{graphicx}

\newcommand{\rra}{\rightrightarrows}

\newcommand{\rha}{\rightharpoonup}
\DeclareMathOperator{\cl}{\mathrm{cl}}
\DeclareMathOperator{\graph}{\mathrm{graph}}
\def\dA{\delta A}
\def\dV{\delta V}

\def\A{{\cal A}}
\def\Ah{\A_k}
\def\U{{\cal U}}
\def\Uh{\U_k}
\def\T{{\cal T}}
\def\Th{\T_k}

\def\div{\nabla\cdot}
\def\grad{\nabla}
\def\hess{\nabla^2}
\def\dA{\delta A}

\def\BV{\mbox{BV}}
\def\TV{\mbox{TV}}

\begin{document}
\title{Contingent Derivatives and regularization for noncoercive inverse problems}
\subtitle{\lowercase{Dedicated to the blessed memory of Jonathan M. Borwein}}
\author{%
    Christian Clason\footnotemark[3]
    \and 
    Akhtar A. Khan\footnotemark[1]
    \and 
    Miguel Sama\footnotemark[9] 
    \and 
    Christiane Tammer\footnotemark[7]%
}  
\date{February 14, 2018}
\maketitle

\footnotetext[1] {Center for Applied and Computational Mathematics, School of Mathematical Sciences, Rochester Institute of Technology, 85 Lomb Memorial Drive, Rochester, New York, 14623, USA. (\email{aaksma@rit.edu})}
\footnotetext[3]  {Faculty of Mathematics, University of Duisburg-Essen, 45117 Essen, Germany.  (\email{christian.clason@uni-due.de})}
\footnotetext[9] {Departamento de Matem\'atica Aplicada, Universidad Nacional de Educaci\'on a Distancia, Calle  Juan del Rosal, 12, 28040 Madrid, Spain. (\email{msama@ind.uned.es})}
\footnotetext[7]{Institute of Mathematics, Martin-Luther-University of Halle-Wittenberg, Theodor-Lieser-Str. 5, D-06120 Halle-Saale, Germany. (\email{christiane.tammer@mathematik.uni-halle.de})}

\begin{abstract}
    We study the inverse problem of parameter identification in non-coercive variational problems that commonly appear in applied models. We examine the differentiability of the set-valued parameter-to-solution map by using the first-order and the second-order contingent derivatives. We explore the inverse problem by using the output least-squares and the modified output least-squares objectives. By regularizing the non-coercive variational problem, we obtain a single-valued regularized parameter-to-solution map and investigate its smoothness and boundedness. We also consider optimization problems using the output least-squares and the modified output least-squares objectives for the regularized variational problem. We give a complete convergence analysis showing that for the output least-squares and the modified output least-squares, the regularized minimization problems approximate the original optimization problems suitably. We also provide the first-order and the second-order adjoint method for the computation of the first-order and the second-order derivatives of the output least-squares objective. We provide discrete formulas for the gradient and the Hessian calculation and present numerical results.

    \vspace*{0.5\baselineskip}

    \noindent{\color{structure}2000 Mathematics Subject Classification}\quad \textsc{35R30, 49N45, 65J20, 65J22, 65M30}
\end{abstract}

\section{Introduction}\label{sec:intro}

Non-coercive variational problems frequently emerge from applied models (see \cite{Goe96}).  However, often less general versions of such models are studied under the coercivity assumption so that useful technical tools can be employed. For instance, in all the available literature on the inverse problem of identifying a variable parameter in elliptic partial differential equations using a variational framework, the bilinear form has always been chosen to be coercive. The coercivity ensures that the variational problem is uniquely solvable and retains stability on the data perturbation. Under coercivity, the parameter-to-solution map is single-valued, well-defined, and infinitely differentiable. Furthermore, coercivity also plays a decisive role in local stability estimates, see \cite{JadKhaSamTam17}. Although the solvability of a noncoercive variational problem can be ensured by other tools (see \cite{BaiGasTom86}), the parameter-to-solution map, in this case, is a set-valued map. Therefore, for parameter identification in noncoercive variational problems, to derive optimality conditions, one has to employ a suitable notion of a derivative of set-valued maps. Therefore, the known techniques need to be altered significantly to cope with the involvement of such technical tools. For an overview of the recent developments in the vibrant and expanding field of inverse problems, the reader is referred to  \cite{AmmGarJou10,BoeUlb15,BoiKal16,Cla12,CroGocJadKhaWin14,EvsMedSmi16,GhoManBir16,GocJadKhaTamWin15,GucBan15,
HagNgoYasZha15,KinMutRes14,KhaMot18,KirPosResSch15,KucSte15,Liu16,ManGun00,NeuHeiHofKinTau10}.

A prototypical example of a non-coercive variational problem is the weak formulation of pure Neumann boundary value problem (BVP): Given a bounded open domain $\Omega$ and the unit outer unit normal $n$, consider the problem of finding $u$ such that
\begin{equation}\label{Neumann}
    -\nabla \cdot(a\nabla u)=f\ \  \mbox{in}\ \Omega,\  \frac{\partial u}{\partial n}=g\ \  \mbox{on}\ \partial\Omega,
\end{equation}
where $\frac{\partial u}{\partial n}$ is the outer normal derivative of $u$ on the boundary $\partial \Omega$, and $f$ and $g$ are two given functions. It is known that the weak form of the above BVP leads to a noncoercive bilinear form. Moreover,  \eqref{Neumann} is solvable only under the compatibility condition
\begin{equation*}
    \int_{\Omega}f+\int_{\partial \Omega}g=0,
\end{equation*}
whereas, as a consequence of Fredholm alternate, there are infinitely many solutions, with any two solutions only differing by a constant. Furthermore, among these solutions, there is a unique solution under the additional constraint $\int_{\Omega}u=0$.

All the research on the inverse problems of parameter identification in pure Neumann BVP,  the constraint $\int_{\Omega}u=0$ and the compatibility condition have been imposed so that the parameter-to-solution map $a\to u(a)$ is well-defined and single-valued.

The primary objective of this work is to conduct a thorough study of the inverse problem of parameter identification in noncoercive variational problems. However, before going into the details of the main contribution of this article, we provide a brief review of the existing approaches for parameter identification in partial differential equations (PDEs) and variational problems by focusing on the role of coercivity. Let $B$ be a Banach space and  let $A$ be a closed, and  convex subset of $B$ with a nonempty interior. Given a Hilbert space  $V$, let $T:B\times V\times V\rightarrow\mathbb{R}$ be a trilinear form with $T(a,u,v)$ symmetric in $u$, $v$, and let $m$ be a bounded linear functional on $V$. Assume there are constants $\alpha >0$ and $\beta >0$ such that the following continuity (cf. \eqref{Cont}) and coercivity (cf. \eqref{Coer}) conditions hold:
\begin{align}
    T(a,u,v)&\le \beta\|a\|_{B}\|u\|_{V}\|v\|_{V},\ \ \text{for all}\ u,v\in V,\ a\in B,\label{Cont}\\
    T(a,u,u)&\ge\alpha\|u\|_{V}^2,\ \text{for all}\ u\in V,\ a\in A. \label{Coer}
\end{align}
Consider the following variational problem: Given $a\in A$, find $u=u(a)\in V$ such that
\begin{equation}\label{VP1}
    T(a,u,v)=m(v),\ \ \text{for every}\ v\in V.
\end{equation}
Due to the conditions imposed on the trilinear map $T$, the Riesz representation theorem ensures that for any $a\in A$, variational problem \eqref{VP1} admits a unique solution $u(a)$ (see \cite{Goe96}). The inverse problem now seeks the parameter $a$ in \eqref{VP1} from a measurement $z$ of $u$. This inverse problem  is often studied in an optimization framework, either formulating the problem as an unconstrained optimization problem or treating it as a
constrained optimization problem in which the variational problem itself is the
constraint.

Given a Banach space $Z\supseteq V$ equipped with the norm $\|\cdot\|_Z$, the most commonly adopted optimization framework minimizes the following output least-squares (OLS) functional
\begin{equation}\label{OLS0}
    J_1(a):=\|u(a)-z\|^2_Z
\end{equation}
where $z\in Z$ is the data (the measurement of $u$) and $u(a)$ solves the
variational form \eqref{VP1}.

For an insight into the abstract framework, consider the boundary value problem (BVP)
\begin{equation}\label{scalar}
    -\nabla \cdot(a\nabla u)=f\ \  \mbox{in}\ \Omega,\  u=0\ \  \mbox{on}\ \partial\Omega,
\end{equation}
where $\Omega$ is a suitable domain in $\mathbb{R}^2$ or $\mathbb{R}^3$ and $\partial
\Omega$ is its boundary. BVP \eqref{scalar} models useful real-world problems and has been studied in detail. For example, in \eqref{scalar}, $u=u(x)$ may
represent the steady-state temperature at a point $x$ of a
body; then $a$ would be a variable thermal conductivity coefficient,
and $f$ the external heat source. BVP \eqref{scalar} also
models underground steady state aquifers in which the parameter $a$ is the aquifer transmissivity coefficient, $u$ is the hydraulic head, and $f$ is the recharge. The inverse problem in the context of \eqref{scalar} is to estimate
the parameter $a$ from a measurement $z$ of the solution $u$.

For \eqref{scalar} with $Z=L_2(\Omega)$, optimization problem \eqref{OLS0} reduces to minimizing \begin{equation}\label{ols}
    \tilde{J}_1(a):=\int_{\Omega}(u(a)-z)^2,
\end{equation}
where $z$ is the measurement of $u$ and $u(a)$ solves the
variational form of \eqref{scalar} given by
\begin{equation}\label{vpp}
    \int_{\Omega}a\nabla u\cdot\nabla v=\int_{\Omega}fv, \ \ \text{for all}\,v\in H_0^1(\Omega).
\end{equation}
A variant of \eqref{OLS0} is the following modified OLS functional (MOLS) introduced in \cite{GocKha07}
\begin{equation} \label{MOLS}
    J_2(a)=T(a,u(a)-z,u(a)-z),
\end{equation}
where $z$ is the data and $u(a)$ solves \eqref{VP1}. In \cite{GocKha07}, the author established that \eqref{MOLS} is convex and used it to estimate the Lam\'e moduli in the equations  of isotropic elasticity. Studies related to MOLS functional and its extensions can be found in \cite{GocJadKha06,GocKha09,JadKhaRusSamWin14}.

The MOLS functional, given in  \eqref{MOLS}, was inspired by Knowles~\cite{Kno01} who minimized  a
coefficient-dependent norm
\begin{equation}\label{ols1}
    \tilde{J}_2(a):=\int_{\Omega} a\nabla (u(a)-z)\cdot\nabla (u(a)-z),
\end{equation}
where $z$ is the measurement of $u$, and $u(a)$ solves
\eqref{vpp}.

Besides the OLS functional  and the MOLS functional, there are other approaches. The equation error method (cf.
\cite{Aca93,AlJGoc12,Kar97}), for \eqref{scalar}, consists of minimizing the functional
\begin{equation*}
    \tilde{J}_3(a):=\|\nabla (a\nabla z)+f\|^2_{H^{-1}(\Omega)}
\end{equation*}
where $H^{-1}(\Omega )$ is the  dual of $H_0^1(\Omega )$ and $z$ is
the data. In \cite{GocJadKha08}, the equation error approach was explored in an abstract framework.

Finally, we recall the following two results from \cite{GocKha07} for the parameter-to-solution map.
\begin{lemma}\label{LCE} For any $a\in A$, the solution $u(a)$ of the variational problem \eqref{VP1} satisfies $\|u(a)\|_V\leq\alpha^{-1}\|m\|_{V^*}$. Moreover, for any $a,b\in A$, we have
    \begin{equation}\label{CE}
        \|u(a)-u(b)\|_V\leq \min\left\{ \frac{\beta}{\alpha}\|u(a)\|_V,\frac{\beta}{\alpha}\|u(b)\|_V,\frac{\beta}{\alpha^2}\|m\|_{V^*}\right\}\|b-a\|_B.
    \end{equation}
\end{lemma}
\begin{lemma}\label{d1} 
    For each $a$ in the interior of $A$, the solution $u(a)$ of the variational problem \eqref{VP1} is infinitely
    differentiable at $a$. Given $u=u(a)$, the first derivative
    $Du(a)\delta a$ of $u(a)$ in the direction $\delta a$, is the unique solution of the following variational equation
    \begin{equation}\label{var4}
        T(a,Du(a)\delta a,v)=-T(\delta a,u,v),\ \forall\,v\in V.
    \end{equation}
    Furthermore, the second derivative $D^2u(a)(\delta a_1,\delta a_2)$ of $u(a)$ in the direction $(\delta a_1,\delta a_2)$, is the unique solution of the
    following variational equation
    \begin{equation}\label{SOD}
        T(a,D^2u(a)(\delta a_1,\delta a_2),v)=-T(\delta a_2,Du(a)\delta a_1,v)-T(\delta a_1,Du(a)\delta a_2,v),\ \forall\,v\in V.
    \end{equation}
    Moreover, the following bounds hold:
    \begin{alignat}{2}
        \|Du(a)\|&\leq &\frac{\beta}{\alpha}\|u(a)\|_V&\leq \frac{\beta}{\alpha^2}\|m\|_{V^*},\label{FOB}\\
        \left\|D^2u(a)\right\|&\leq &\,\frac{2\beta^2}{\alpha^2}\|u(a)\|_V&\leq\frac{2\beta^2}{\alpha^3}\|m\|_{V^*}.\label{SOB}
    \end{alignat}
\end{lemma}

In all of the above results, coercivity condition \eqref{Coer} played the most crucial role. It gives the unique solvability of variational problem \eqref{VP1}, proves bound on the parameter-to-solution map, establishes its Lipschitz continuity and infinite differentiability. As another useful consequences of the coercivity, the first-order and the second-order derivatives of the parameter-to-solution maps are the unique solutions of the variational problems \eqref{var4} and \eqref{SOD}. Moreover, the useful bounds \eqref{FOB} and \eqref{SOB} also hold due to the coercivity.

This work aims to study the inverse problem of parameter identification in noncoercive variational problems with perturbed data.  Our main contributions are as follows:
\begin{enumerate}
    \item Assuming that the noncoercive variational problem is solvable,
        we give a derivative characterization for the set-valued parameter-to-solution map by using the first-order and the second-order contingent derivatives. To our knowledge, this is the first use of such tools from variational analysis in the study of inverse problems.
    \item We study the inverse problem by posing optimization problems using the output least-squares and the modified output least-squares functionals for the set-valued parameter to solution map. We regularize the noncoercive variational problem and obtain the single-valued regularized parameter-to-selection map and explore its smoothness. We consider optimization problems using the output least-squares and the modified output least-squares for the regularized variational problem. We prove that the MOLS objective is convex and give a complete convergence analysis showing that the regularized problems approximate the original problem suitably.
    \item To compute the first-order and the second-order derivative of the OLS functional, we give first-order, and second-order adjoint methods in the continuous setting. We provide a discretization scheme and give discrete formulas for the OLS and the MOLS functionals and their gradient and Hessian calculation. As a byproduct of our study, we obtain new insight into the case when the actual trilinear form is coercive, however, for the computations, only its contaminated analog is available which is noncoercive. All the conditions imposed for the convergence
        analysis are satisfied in this case of practical importance.
\end{enumerate}

We organize the contents of this paper into seven sections. \Cref{sec:solution-map} introduces the inverse problem and explores the smoothness of the set-valued parameter-to-solution map. \Cref{sec:framework} investigates the inverse problem by using the output-least squares approach and the modified output least squares approach. \Cref{sec:adjoint} is devoted to the first-order and the second-order adjoint approaches. \Cref{sec:computational} provides a detailed computational framework including the discrete gradient and Hessian formulae. In \cref{sec:experiments}, we report the outcome of some preliminary numerical experiments. The paper concludes with some remarks.

\section{Set-valued Solution Map for a Noncoercive Variational Problem}\label{sec:solution-map}

For convenience, we recall the general setting once again. Let $B$ be a Banach space, let $A\subset B$ be a nonempty, closed, and  convex set. Let $V$ be a Hilbert space continuously imbedded in a Hilbert space $Z$. Let $T:B\times V\times V\rightarrow\mathbb{R}$ be a trilinear form with $T(a,u,v)$ symmetric in $u$, $v$.  Let $m$ be a bounded linear functional on $V$. Assume that $T$ satisfies the continuity assumption \eqref{Cont} and the following positivity condition:
\begin{equation}\label{Pos}
    T(a,u,u)\geq 0,\ \text{for all}\ u\in V,\ a\in A.
\end{equation}
Consider the noncoercive variational problem: Given $a\in A$, find $u=u(a)\in V$ such that
\begin{equation}\label{VP}
    T(a,u,v)=m(v),\ \ \text{for every}\ v\in V.
\end{equation}
Since we do not impose the coercivity condition (see \eqref{Coer}) on $T$, additional conditions are necessary to ensure that \eqref{VP} is solvable. For example, recession analysis can be used to ensure  that \eqref{VP} is solvable but such conditions don't guarantee that the solution is unique (see \cite{Goe96}). Therefore, it is natural to study the behavior of the set-valued parameter-to-solution map.  For a given parameter $a\in A$, by $\mathcal{U}(a)$ we denote the set of all solutions of variational equation \eqref{VP}. In the following, we assume that for each $a\in A$, the set $\mathcal{U}(a)$ is nonempty. The following lemma provides additional information:
\begin{lemma} 
    For any $a\in A$, the set of all solutions $\mathcal{U}(a)$ of \eqref{VP} is closed and convex.
\end{lemma}
\begin{proof} The proof follows at once from the definition of the set-valued parameter-to-solution map $\mathcal{U}:A\rra V$. Indeed, let $u$ and $w$ be two arbitrary elements in $\mathcal{U}(a)$. Then, for every $v\in V$, we have $T(a,u,v)=m(v)$ and $T(a,w,v)=m(v)$. We take $t\in [0,1]$ and note that for every $v\in V$, we have $T(a,tu,v)=tm(v)$ and $T(a,(1-t)u,v)=(1-t)m(v)$. We combine these equations to note that for every $v\in V$, we have $T(a,tu+(1-t)w,v)=m(v)$. Consequently, $tu+(1-t)w\in \mathcal{U}(a)$ confirming the convexity of $\mathcal{U}(a)$. The set $\mathcal{U}(a)$ is closed due to the continuity of the trilinear form $T$. The proof is complete.
\end{proof}

\bigskip

Our goal is to obtain a derivative characterization for the set-valued parameter-to-solution map. In the literature, a variety of derivative concepts have been employed to differentiate set-valued maps (see \cite{KhaTamZal15}). We will use first-order and second-order contingent derivatives of the parameter-to-solution set-valued map $\mathcal{U}:A\rightrightarrows V$. These derivatives are defined by using the contingent cone and the second-order contingent set which we recall now.
\begin{definition}\label{d2.1}Let $X$ be a normed space, let $S \subset X$ and let $\bar{z}\in
    \cl(S)$ (closure of $S$).
    \begin{enumerate}
        \item The \emph{contingent cone} $C(S,\bar{z})$ of $S $ at $\bar{z}$ is  the set of all $z \in
            X$ such that there are sequences $\{t_{n}\}\subset \mathbb{P}:=\{t\in \mathbb{R}\,|\ t>0\}$ and
            $\{z_{n}\}\subset X$ with $t_{n}\downarrow 0$ and  $z_{n}\to z$ such that $\bar{z}+t_{n}z_{n}\in S$, for every $n\in \mathbb{N}$.
        \item The \emph{second-order contingent set} $C^{2}(S,\bar{z},w)$ of $S $ at
            $\bar{z}\in \cl(S)$ in the direction $w\in X $ is the
            set of all $z \in S $ such that there are a sequence $\{z_{n}\}\subset X$ with $z_{n}\to z$ and a sequence $\{t_n\}\subset \mathbb{P}$ with $t_{n}\downarrow 0$ such that $\bar{z}+t_{n} w+t_n^2 z_{n}/2\in S$, for every $n\in \mathbb{N}$.
    \end{enumerate}
\end{definition}
\begin{remark}\label{re2.1}
    It is known that the contingent cone $C(S,\bar{z})$ is a nonempty closed cone. However, $C^{2}(S,\bar{z},w)$ is only a closed set (possibly empty), non-connected in
    general, and it may be nonempty only if $w\in C(S,\bar{z})$. Details of these concepts can be found in \cite{Bor78,Bor76,KhaTamZal15}.
\end{remark}

Next we collect some notions for set-valued maps. Given normed spaces $X$ and $Y$,
let $F: X \rightrightarrows  Y$ be a set-valued map. The (effective) domain and the graph of $F$ are
defined by $\text{dom} (F): = \{ x \in X |\quad F(x)\not = \emptyset\}$, and $\text{graph} (F):= \{ (x,y)\in
X \times Y |\quad y \in F(x) \}$.

We now introduce first-order and second-order derivatives of set-valued maps.
\begin{definition}\label{d4.1} 
    Let $X$ and $Y$ be normed spaces,  let $F:X\rightrightarrows Y$ be a set-valued map, and let $(\bar{x},\bar{y})\in \graph(F)$. Then the \emph{contingent derivative} of $F$ at $(\bar{x},\bar{y})$ is the set-valued map $DF(\bar{x},\bar{y}):X\rightrightarrows  Y$ given by
    \begin{equation*}
        DF(\bar{x},\bar{y})(x):=\{y\in Y\,|\ (x,y)\in C(\graph(F),(\bar{x},\bar{y}))\}.
    \end{equation*}
    Moreover, the \emph{second-order contingent derivative} of $F$ at $(\bar{x},\bar{y})$ in the
    direction $(\bar{u},\bar{v})$ is the set-valued map $D^{2}F(\bar{x},\bar{y},\bar{u},\bar{v}):X
    \rightrightarrows  Y$ defined by
    \begin{equation*}
        D^{2}F(\bar{x},\bar{y},\bar{u},\bar{v})(x):=\left\{y\in Y
        \,|\ (x,y)\in C^{2}(\graph(F), (\bar{x},\bar{y}),(\bar{u},\bar{v}))\right\}.
    \end{equation*}
\end{definition}

The above derivatives have been used extensively in nonsmooth and variational analysis, viability theory, set-valued optimization and numerous other related disciplines, see \cite{KhaTamZal15}.

\bigskip

We have the following derivative characterization for the parameter-to-solution map:
\begin{theorem}\label{T-FCD} For $\bar{a}\in A$, let $\bar{u}\in \mathcal{U}(\bar{a})$ be a given point. Assume that the first-order contingent derivative $D\mathcal{U}(\bar{a},\bar{u}):B\rra V$ of the set-valued parameter-to-solution map $\mathcal{U}:A\rra V$ at the point  $(\bar{a},\bar{u})\in \text{graph}(\mathcal{U})$ exists. Then for any given direction $\delta a\in B$, any element $\delta u\in D\mathcal{U}(\bar{a},\bar{u})(\delta a)$ satisfies the following variational problem:
    \begin{equation}\label{FCD}T(\bar{a},\delta u,v)=-T(\delta a,\bar{u},v),\quad \text{for every}\ v\in V.
    \end{equation}
\end{theorem}
\begin{proof} For the given element $(\bar{a},\bar{u})\in \text{graph}(\mathcal{U})$ and the given direction $\delta a$, for any $\delta u\in D\mathcal{U}(\bar{a},\bar{u})(\delta a)$, we have
    \begin{equation*}
        (\delta a,\delta u)\in \text{graph}(D\mathcal{U}(\bar{a},\bar{u}))=C(\text{graph}(\mathcal{U}),(\bar{a},\bar{u})),
    \end{equation*}
    and by the definition of the contingent cone, there are sequences $\{t_n\}\subset \mathbb{P}$ and $\{(a_n,u_n)\}$ with $t_n\to 0$ and $(a_n,u_n)\to (\delta a,\delta u)$ such that
    $(\bar{a}+t_na_n,\bar{u}+t_nu_n)\in \text{graph}(\mathcal{U})$, or equivalently $\bar{u}+t_nu_n\in \mathcal{U}(\bar{a}+t_na_n)$,
    which, by the definition of the map $\mathcal{U}:A\rightrightarrows V$, implies that
    \begin{equation*}
        T(\bar{a}+t_na_n,\bar{u}+t_nu_n,v)=m(v),\quad \text{for every}\ v\in V,
    \end{equation*}
    and after a rearrangement of this variational problem, we obtain
    \begin{equation*}
        T(\bar{a},\bar{u},v)+t_nT(\bar{a},u_n,v)+t_nT(a_n,\bar{u},v)+t_n^2T(a_n,u_n,v)=m(v),\quad \text{for every}\ v\in V.
    \end{equation*}
    The condition $(\bar{a},\bar{u})\in \text{graph}(\mathcal{U})$ implies that $T(\bar{a},\bar{u},v)=m(v)$, for every $v\in V$, and hence
    \begin{equation*}
        T(\bar{a},u_n,v)=-T(a_n,\bar{u},v)-t_nT(a_n,u_n,v),\quad \text{for every}\ v\in V.
    \end{equation*}
    By passing the above equation to the limit $n\to \infty$, we obtain
    \begin{equation*}
        T(\bar{a},\delta u,v)=-T(\delta a,\bar{u},v),\quad \text{for every}\ v\in V,
    \end{equation*}
    and the desired identity \eqref{FCD} is proved. The proof is complete.
\end{proof}
\begin{remark}
    If the trilinear form $T$ satisfies coercivity condition \eqref{Coer}, then for every parameter $\bar{a}\in A$, variational problem \eqref{VP} has a unique solution $\bar{u}=u(\bar{a})$, that is, the map $\bar{a}\to u(\bar{a})$ is well-defined and single-valued. Moreover, for any $a$ in the interior of $A$ and any direction $\delta a$, the Fr\'echet derivative $\delta u=D\bar{u}(\bar{a})(\delta a)$ is the unique solution of the variational problem \eqref{var4} which is entirely comparable to the characterization \eqref{FCD}.
\end{remark}

\bigskip

The following is the characterization of the second-order contingent derivative:
\begin{theorem}\label{T-SCD} 
    For any $\bar{a}\in A$, let $\bar{u}\in \mathcal{U}(\bar{a})$ be a given element. Assume that second-order contingent derivative of the parameter-to-solution set-valued map $\mathcal{U}:A\rra V$ at $(\bar{a},\bar{u})$ in the direction $(\delta a,\delta u)\in \text{graph}(D\mathcal{U}(\bar{a},\bar{u}))$ exists. Then for any given direction $\delta \tilde{a}\in B$, any element $\delta^2 u\in D^2\mathcal{U}(\bar{a},\bar{u},\delta a,\delta u)(\delta \tilde{a})$ satisfies the variational problem:
    \begin{equation}\label{SCD}
        T(\bar{a},\delta^2 u,v)=-2T(\delta a, \delta u,v)-T(\delta \tilde{a},\bar{u},v),\quad \text{for every}\ v\in V.
    \end{equation}
\end{theorem}
\begin{proof} 
    For the given $(\bar{a},\bar{u})\in \text{graph}(\mathcal{U})$ and the given $(\delta a,\delta u)\in \text{graph}(D\mathcal{U}(\bar{a},\bar{u}))$, let $\delta^2 u\in D^2\mathcal{U}(\bar{a},\bar{u},\delta a,\delta u)(\delta \tilde{a})$. Then, we have
    \begin{equation*}
        (\delta \tilde{a},\delta^2 u)\in \text{graph}(D^2\mathcal{U}(\bar{a},\bar{u},\delta a,\delta u))=C^2(\text{graph}(\mathcal{U}),\bar{a},\bar{u},\delta a,\delta u).
    \end{equation*}
    Therefore, there are sequences $\{t_n\}\subset \mathbb{P}$ and $\{(a_n,u_n)\}\in \text{graph}(\mathcal{U})$ with
    $t_n\to 0$, and $(a_n,u_n)\to (\delta \tilde{a},\delta^2 u)$ so that $(\bar{a}+t_n\delta a+\frac12t_n^2 a_n,\bar{u}+\frac12 t_n\delta u+t_n^2u_n)\in \text{graph}(\mathcal{U})$.
    By the definition of the parameter-to-solution map, we have
    \begin{equation*}
        T(\bar{a}+t_n\delta a+\frac12t_n^2 a_n,\bar{u}+t_n\delta u+\frac12t_n^2u_n,v)=m(v),\quad \text{for every}\ v\in V.
    \end{equation*}
    We simplify the above identity as follows
    \begin{multline}
        T(\bar{a},\bar{u},v)+t_nT(\bar{a},\delta u,v)+\frac12t_n^2T(\bar{a},u_n,v)+t_nT(\delta a,\bar{u},v)+t_n^2T(\delta a, \delta u,v)\\
        +\frac12t_n^3T(\delta a,u_n,v)+\frac12t_n^2T(a_n,\bar{u},v)+\frac12t_n^3T(a_n,\delta u,v)+\frac14t_n^4T(a_n,u_n,v)=m(v),
    \end{multline}
    which, first by using the fact that $(\bar{a},\bar{u})\in \text{graph}(\mathcal{U})$, and then by dividing both sides of the resulting identity by $t_n$ confirms that
    \begin{multline*}
        T(\bar{a},\delta u,v)+\frac12t_nT(\bar{a},u_n,v)+T(\delta a,\bar{u},v)+t_nT(\delta a, \delta u,v)\\
        +\frac12t_n^2T(\delta a,u_n,v)+\frac12t_nT(a_n,\bar{u},v)+\frac12t_n^2T(a_n,\delta u,v)+\frac14t_n^3T(a_n,u_n,v)=0,
    \end{multline*}
    We now first use the fact that $(\delta a,\delta u)\in \text{graph}(D\mathcal{U}(\bar{a},\bar{u}))$, and then divide both sides of the resulting identity by $t_n$ to obtain
    \begin{multline*}
        \frac12 T(\bar{a},u_n,v)+T(\delta a, \delta u,v)+\frac12 t_nT(\delta a,u_n,v)\\
        +\frac12 T(a_n,\bar{u},v)+\frac12t_nT(a_n,\delta u,v)+\frac14 t_n^2T(a_n,u_n,v)=0,
    \end{multline*}
    which when passed to the limit $t_n\to 0$, yields
    \begin{equation*}
        T(\bar{a},\delta^2 u,v)=-2T(\delta a, \delta u,v)-T(\delta \tilde{a},\bar{u},v),
    \end{equation*}
    proving \eqref{SCD}. The proof is complete.
\end{proof}

\bigskip

Note that due to the characterization of the first-order contingent derivative of the set-valued map $\mathcal{U}:A\rra V$, variational problem \eqref{SCD} is equivalent to
\begin{equation}\label{SCD1}
    T(\bar{a},\delta^2 u,v)=-2T(\delta a, \delta u,v)+T(\bar{a},\delta \tilde{u},v),
\end{equation}
where $\delta \tilde{u}\in D\mathcal{U}(\bar{a},\bar{u})(\delta \tilde{a})$ and $v\in V$ is arbitrary.

\bigskip

Clearly, if $T$ satisfies condition \eqref{Coer}, then the parameter-to-solution map is single-valued and infinitely differentiable in the interior of the domain. Moreover, for any  $\bar{a}$ in the interior of $A$, and suitable directions 
$(\delta a_1,\delta a_2)$, the second-order derivative $D^2u(\bar{a})(\delta a_1,\delta a_2)$ is the unique solution of \eqref{SOD}. In particular, with $\delta a=\delta a_1=\delta a_2$, we have
\begin{equation}\label{SOD1}
    T(\bar{a},D^2u(\bar{a})(\delta a,\delta a),v)=-2T(\delta a,Du(a)\delta a,v),\ \ \text{for every}\ v\in V.
\end{equation}
We also recall that if a single-valued map $F:X\to Y$ is twice differentiable, then with $DF(x)$ and $D^2F(x)$ as the first-order and the second-order derivatives,  we have (see \cite{War93})
\begin{equation*}
    C^2(\text{graph}(F),(x,F(x),v,DF(x)v))=\{(y,DF(x)y+D^2F(x)(v,v)),\ y\in X\}.
\end{equation*}
Consequently, by taking $\delta a=\delta \tilde{a}$, we have
\begin{equation*}
    \text{graph}(D^2\mathcal{U}(\bar{a},\bar{u},\delta a,\delta u))=\{(\delta a, D\bar{u}(\bar{a})\delta a+D^2\bar{u}(\bar{a}) (\delta a,\delta a)),\ \delta a\in B\}
\end{equation*}
and, as a result, under \eqref{Coer}, the derivative formula yields
\begin{equation*}
    T(\bar{a},D\bar{u}(\bar{a})\delta a+D^2\bar{u}(\bar{a})(\delta a,\delta a),v)=-2T(\delta a, D\bar{u}(\bar{a})(\delta a),v)+T(\bar{a},D\bar{u}(\bar{a})\delta a,v),\ \text{for all}\ v\in V,
\end{equation*}
implying that
\begin{equation*}
    T(\bar{a},D^2\bar{u}(\bar{a}) (\delta a,\delta a),v)=-2T(\delta a, D\bar{u}(\delta)(\delta a),v),\quad \text{for every}\ v\in V,
\end{equation*}
which is in compliance with the second-order formula \eqref{SOD1}.
\begin{remark}The results given above only offer characterizations of the first-order and the second-order contingent derivatives under the critical assumption that these derivatives exist. This is a natural step as we have not identified conditions under which the variational problem is solvable. A possible extension of these results is singling out conditions ensuring the existence of solutions and then using them to verify the contingent differentiability. 
\end{remark}

\section{Recasting the Inverse Problem in an Optimization Framework}\label{sec:framework}

\subsection{The Output Least-Squares Approach}\label{sec:framework:ols}

Let $\mathcal{U}:A\rra V$ be the set-valued parameter-to-solution map which assigns to each $a\in A$, the set of all solutions $\mathcal{U}(a)$ of the noncoercive variational problem \eqref{VP}.  We define the set-valued output least-squares map $\widehat{J}:A\rra \mathbb{R}$ which connects to each $a\in A$, the following set
\begin{equation*}
    \widehat{J}(a):=\left\{\|u(a)-z\|^2_Z\mid\ u(a)\in \mathcal{U}(a)\right\},
\end{equation*}
where $z\in Z$ is the measured data.

Using the above set-valued map (and a slight abuse of the notation), we pose the following OLS-based optimization problem
\begin{equation}\label{MinOLS}  \min_{a\in A}\widehat{J}(a).
\end{equation}
The philosophy of the OLS approach is  to minimize the gap between the computed solutions $u(a)\in \mathcal{U}(a)$ of \eqref{VP}  and the measured data $z\in Z$.

An element $\bar{a}\in A$ is called a minimizer of \eqref{MinOLS}, if there exists $u(\bar{a})\in \mathcal{U}(\bar{a})$ such that
\begin{equation}\label{DefMin}
    \|u(\bar{a})-z\|^2_Z\leq \|u(a)-z\|^2_Z,\quad \text{for every}\ u(a)\in \mathcal{U}(a),\ \text{for every}\ a\in A. 
\end{equation}
To emphasize the role of $u(\bar{a})$, we sometimes say that $(\bar{a},u(\bar{a}))\in \text{graph}(\mathcal{U})$ is a minimizer.

\bigskip

One of our goals is to approximate \eqref{MinOLS} by a sequence of solutions of optimization problems for which the entire data set of the constraint variational problem is noisy in the sense described below. Let $\{\epsilon_n\}$, $\{\tau_n\}$, $\{\kappa_n\}$, $\{\delta_n\}$, and $\{\nu_n\}$ be sequences of positive reals.
Let $\ell\in V^*$ be a given element.  For each $n\in \mathbb{N}$, let $m_{\nu_n}\in V^*$ and $\ell_{\delta_n}\in V^*$ be given elements, and let $z_{\delta_n}\in  Z$ be the contaminated data such that the following inequalities hold:
\begin{subequations}\label{Noise}
    \begin{align}
        \|z_{\delta_n}-z\|_Z&\leq \delta_n,\label{Z-Error}\\
        \|m_{\nu_n}-m\|_{V^*}&\leq \nu_n,\label{M-Error}\\
        \|\ell_{\delta_n}-\ell\|_{V^*}&\leq \delta_n.\label{L-Error}
    \end{align}
\end{subequations}
Furthermore, for each $n\in \mathbb{N}$, let $T_{\tau_n}:B\times V\times V\to \mathbb{R}$ be a trilinear form such that
\begin{subequations}\label{TNoise}
    \begin{align}
        T_{\tau_n}(a,u,u)&\ge  0,\ \text{for all}\ u\in V,\ a\in A. \label{Pos1}\\
        \left|T_{\tau_n}(a,u,v)-T(a,u,v)\right|&\leq \tau_n\|a\|_B\|u\|_V\|v\|_V,\ \ \text{for all}\ u,v\in V,\ a\in B.\label{TError}
    \end{align}
\end{subequations}
Moreover, as $n\to \infty$, the sequences $\{\epsilon_n\}$, $\{\tau_n\}$,  $\{\kappa_n\}$, $\{\delta_n\}$, and $\{\nu_n\}$ satisfy
\begin{equation}\label{SeqC}
    \left\{\tau_n,\epsilon_n,\kappa_n,\nu_n,\delta_n,\frac{\tau_n}{\epsilon_n},\frac{\delta_n}{\epsilon_n},\frac{\nu_n}{\epsilon_n}\right\}\to 0.
\end{equation}
Finally, let $S:V\times V\to \mathbb{R}$ be a symmetric bilinear form such that there are constants $\alpha_0 >0$ and $\beta_0 >0$ satisfying the following continuity and coercivity conditions
\begin{subequations}\label{SCond}
    \begin{align}
        S(u,v)&\le \beta_0\|u\|_{V}\|v\|_{V},\ \ \text{for all}\ u,v\in V,\label{Cont0}\\
        S(u,u)&\ge\alpha_0\|u\|_{V}^2,\ \text{for all}\ u\in V. \label{Coer0}
    \end{align}
\end{subequations}

With the above preparation, for each $n\in \mathbb{N}$, we now consider the following regularized variational problem: Given $a\in A$, find $u_{\varsigma_n}(a)\in V$ such that
\begin{equation}\label{RVP}
    T_{\tau_n}(a,u_{\varsigma_n}(a),v)+\epsilon_n S(u_{\varsigma_n}(a),v)=m_{\nu_n}(v)+\epsilon_n\ell_{\delta_n}(v),\ \ \text{for every}\ v\in V.
\end{equation}
where $\epsilon_n>0$ is the regularization parameter and for simplicity, we set $\varsigma_n:=(\epsilon_n,\tau_n,\nu_n,\delta_n)$.

In view of the above conditions, for a fixed $n\in \mathbb{N}$, and for every $a\in A$, \eqref{RVP} has a unique solution  $u_{\varsigma_n}(a)$. Therefore, the regularized parameter-to-solution map $a\to u_{\varsigma_n}(a)$ is well-defined and single-valued.

The next result embarks on the smoothness of the regularized parameter-to-solution map:
\begin{theorem}\label{d1RP} For any $n\in \mathbb{N}$ and any parameter $a$ in the interior of $A$, the regularized  parameter-to-solution map $a\to u_{\varsigma_n}(a)$ is infinitely
    differentiable at $a$. Moreover, given $u_{\varsigma_n}(a)$, the first-order derivative
    $Du_{\varsigma_n}(a)\delta a$ in the direction $\delta a\in B$ is the unique solution of the variational equation
    \begin{equation}\label{var4RP}
        T_{\tau_n}(a,Du_{\varsigma_n}(a)\delta a,v)+\epsilon_n S(Du_{\varsigma_n}(a)\delta a,v) =-T_{\tau_n}(\delta a,u_{\varsigma_n},v),\ \text{for every}\ v\in V,
    \end{equation}
    and the second-order derivative $D^2u_{\varsigma_n}(a)(\delta a_1,\delta a_2)$ in the direction $(\delta a_1,\delta a_2) \in B\times B$ is the unique solution of the
    variational equation
    \begin{multline}\label{SODRP}
        T_{\tau_n}(a,D^2u_{\varsigma_n}(a)(\delta a_1,\delta a_2),v)+\epsilon_n S(D^2u_{\varsigma_n}(a)(\delta a_1,\delta a_2),v)\\
        =-T_{\tau_n}(\delta a_2,Du_{\varsigma_n}(a)\delta a_1,v)-T_{\tau_n}(\delta a_1,Du_{\varsigma_n}(a)\delta a_2,v),\ \text{for every}\ v\in V.
    \end{multline}
\end{theorem}
\begin{proof} 
    The proof follows by similar arguments that were used in the proof of  \cref{d1}. The crucial role in the proof  is played by the ellipticity of $T_{\tau_n}+\epsilon_n S$.
\end{proof}

\bigskip

Before any further advancement, in the following result we give necessary conditions ensuring that the derivative of the regularized parameter-to-solution map remains bounded.
\begin{theorem}\label{T-FCDB} 
    For a parameter $\bar{a}$ in the interior of $A$, let $\bar{u}\in \mathcal{U}(\bar{a})$ be a given point. Assume that the first-order contingent derivative of the set-valued parameter-to-solution map $\mathcal{U}:A\rra V$ at the point  $(\bar{a},\bar{u})\in \text{graph}(\mathcal{U})$ exists. If
    \begin{equation}\label{Disc}\|u_{\varsigma_n}(\bar{a})-\bar{u}\|_V=O(\epsilon_n),
    \end{equation} where $u_{\varsigma_n}(\bar{a})$ is the regularized solution of \eqref{RVP} for parameter $\bar{a}$, then the first-order derivative $Du_{\varsigma_n}(\bar{a})\delta a$ of $u_{\varsigma_n}(\bar{a})$ in any direction $\delta a\in B$ is uniformly bounded.
\end{theorem}
\begin{proof} 
    From \cref{T-FCD}, for any $\delta a\in B$, and any  $\delta u\in D\mathcal{U}(\bar{a},\bar{u})(\delta a)$, we have
    \begin{equation}\label{DF1}
        T(\bar{a},\delta u,v)=-T(\delta a,\bar{u},v),\quad \text{for every}\ v\in V.
    \end{equation}
    Furthermore, due to \cref{d1RP}, we also have
    \begin{equation}\label{DF}
        T_{\tau_n}(\bar{a},Du_{\varsigma_n}(\bar{a})\delta a,v)+\epsilon_n S(Du_{\varsigma_n}(\bar{a})\delta a,v) =-T_{\tau_n}(\delta a,u_{\varsigma_n}(\bar{a}),v),\ \text{for every}\ v\in V.
    \end{equation}
    We subtract \eqref{DF1} from \eqref{DF} and rearrange the resulting equation to obtain
    \begin{multline*}
        T(\bar{a},Du_{\varsigma_n}(\bar{a})\delta a-\delta u,v)+\epsilon_n S(Du_{\varsigma_n}(\bar{a})\delta a,v)
        =T(\delta a,\bar{u}-u_{\varsigma_n}(\bar{a}),v)+T(\delta a,u_{\varsigma_n}(\bar{a}),v)\\
        -T_{\tau_n}(\delta a,u_{\varsigma_n}(\bar{a}),v)+T(\bar{a},Du_{\varsigma_n}(\bar{a})\delta a,v )-T_{\tau_n}(\bar{a},Du_{\varsigma_n}(\bar{a})\delta a,v ).
    \end{multline*}
    By setting $v=Du_{\varsigma_n}(\bar{a})\delta a-\delta u$ and using the positivity of $T$, we obtain
    \begin{equation*}
        \begin{aligned}
            \epsilon_n\|Du_{\varsigma_n}&(\bar{a})\delta a-\delta u\|_V^2\\
                                        &\leq T(\bar{a},Du_{\varsigma_n}(\bar{a})\delta a-\delta u,Du_{\varsigma_n}(\bar{a})\delta a-\delta u)+\epsilon_n S(Du_{\varsigma_n}(\bar{a})\delta a-\delta u,Du_{\varsigma_n}(\bar{a})\delta a-\delta u)\\
                                        &=T(\delta a,\bar{u}-u_{\varsigma_n}(\bar{a}),Du_{\varsigma_n}(\bar{a})\delta a-\delta u)+T(\delta a,u_{\varsigma_n}(\bar{a}),Du_{\varsigma_n}(\bar{a})\delta a-\delta u)\\
            \MoveEqLeft[-1]-T_{\tau_n}(\delta a,u_{\varsigma_n}(\bar{a}),Du_{\varsigma_n}(\bar{a})\delta a-\delta u)+T(\bar{a},Du_{\varsigma_n}(\bar{a})\delta a,Du_{\varsigma_n}(\bar{a})\delta a-\delta u )\\
            \MoveEqLeft[-1]-T_{\tau_n}(\bar{a},Du_{\varsigma_n}(\bar{a})\delta a,Du_{\varsigma_n}(\bar{a})\delta a-\delta u )-\epsilon_n S(\delta u,Du_{\varsigma_n}(\bar{a})\delta a-\delta u),
        \end{aligned}
    \end{equation*}
    which, due to the properties of $T$ and $S$, implies that
    \begin{multline*}
        \epsilon_n\|Du_{\varsigma_n}(\bar{a})\delta a-\delta u\|_V\leq \beta \|\delta a\|_B\|\bar{u}-u_{\varsigma_n}(\bar{a})\|_V+\tau_n\|\delta a\|_B\|u_{\varsigma_n}(\bar{a})\|_V\\
        +\tau_n\|\bar{a}\|_B\|Du_{\varsigma_n}(\bar{a})\delta a\|_V+\beta_0\epsilon_n\|\delta u\|,
    \end{multline*}
    and hence
    \begin{multline*}
        \left(1-\frac{\tau_n}{\epsilon_n}\|\bar{a}\|_B\right)\|Du_{\varsigma_n}(\bar{a})\delta a-\delta u\|_V\leq \left[\beta+\tau_n\right]\|\delta a\|_B\frac{\|u_{\varsigma_n}(\bar{a})-\bar{u}\|}{\epsilon_n}+\frac{\tau_n}{\epsilon_n}\|\delta a\|_B\|\bar{u}\|\\
        +\beta_0\|\delta u\|+\frac{\tau_n}{\epsilon_n}\|\bar{a}\|_B\|\delta u\|_V.
    \end{multline*}
    In view of \eqref{SeqC} and the assumption that $\|u_{\varsigma_n}(\bar{a})-\bar{u}\|=O(\epsilon_n)$, it follows that there is a constant $c>0$ such that
    \begin{equation*}
        \left(1-\frac{\tau_n}{\epsilon_n}\|\bar{a}\|_B\right)\|Du_{\varsigma_n}(\bar{a})\delta a-\delta u\|_V\leq c,
    \end{equation*}
    and since $\frac{\tau_n}{\epsilon_n}\to 0$ as $n\to \infty$, for sufficiently large $n\in \mathbb{N}$, we have $\left(1-\frac{\tau_n}{\epsilon_n}\|\bar{a}\|_B\right)>0$ and the boundedness of
    $\|Du_{\varsigma_n}(\bar{a})\delta a-\delta u\|_V$ follows. The proof is complete.
\end{proof}
\begin{remark}
    The fundamental idea of the elliptic regularization for variational problems is to combine the variational problem with the regularized analog to ensure that the regularized solutions remain bounded (see \cite{KhaTamZal15b}). Then a subsequence can be extracted and shown to converge weakly  to a solution of the variational problem, ensuring its solvability. In the above result, we use this idea to prove the boundedness of the derivatives. In the present context, the role of the original variational problem is played by the derivative characterization involving the first-order contingent derivative. As a consequence, suitable conditions ensuring the boundedness of the derivatives of the regularized parameter-to-solution map can be used to show the contingent differentiability of the set-valued parameter-to-solution map. We also note that if the original trilinear map is elliptic, then \eqref{Disc} holds.
\end{remark}

\bigskip

To incorporate regularization in the ill-posed inverse problem, we assume:
\begin{enumerate}
    \item The Banach space $B$  is continuously embedded in a Banach space $L$. There is another Banach space $\widehat{B}$ that is compactly embedded in $L$. The set $A$ is a subset of $B\cap \widehat{B}$, closed and bounded in $B$ and also closed in $L$.
    \item $R:\widehat{B}\to \mathbb{R}$ is positive, convex, and  lower-semicontinuous in $\|\cdot\|_L$ such that
        \begin{equation}\label{RegC}
            R(a)\geq \tau_1\|a\|_{\widehat{B}}-\tau_2, \quad \text{for every}\ a\in A,\quad \text{for some}\ \tau_1>0,\ \tau_2>0.
        \end{equation}
    \item For any $\{b_{k}\}\subset B$ with  $b_{k}\rightarrow 0$
        in $L$, any bounded  $\{u_{k}\}\subset V$, and fixed $v\in V$, we have
        \begin{equation}\label{TConC}
            T(b _{k},u_{k},v)\rightarrow 0.
        \end{equation}
\end{enumerate}

The above framework is inspired by the use of total variation regularization in the identification of discontinuous coefficients.  Recall that the total variation of
$f\in L^1(\Omega)$ reads
\begin{equation*}
    \TV(f)=\sup\left\{\int_{\Omega}f\,(\div g)\,:\,g\in \left(C^1_0(\Omega)\right)^N,\
    |g(x)|\le 1\text{ for all } x\in\Omega\right\}
\end{equation*}
where $|\cdot|$ is the Euclidean norm.  Clearly, if $f\in W^{1,1}(\Omega)$, then $\TV(f)=\int_{\Omega}|\grad f|$.

If $f\in L^1(\Omega)$ satisfies $\TV(f)<\infty$, then $f$ is said to have
bounded variation, and $\BV(\Omega)$ is defined by
$\BV(\Omega)=\left\{f\in L^1(\Omega)\,:\,\TV(f)<\infty\right\}$ with norm $\|f\|_{\text{\tiny{BV}}(\Omega)}=\|f\|_{L^1(\Omega)}+\TV(f)$.
The functional $\TV(\cdot)$ is a seminorm on $\BV(\Omega)$ and is often called
the BV-seminorm, see \cite{NasSch98}.

We set $B=L^{\infty}(\Omega)$, $L=L^{1}(\Omega)$,
$\widehat{B}=\BV(\Omega)$, and $R(a)=TV(a)$, and define
\begin{equation}\label{SetA}
    A=\{a\in L^{\infty}|\ 0<c_1\leq a(x)\leq c_2,\ a.e.\ \text{in}\ \Omega,\ \TV(a)\leq c_3<\infty \},
\end{equation}
where $c_1,c_2$, and $c_3$ are positive constants. Clearly, $ A$ is bounded in
$\|\cdot\|_{\widehat{B}}$ and compact in $L$. It is known that $L^{\infty}(\Omega)$ is continuously embedded in $L^1(\Omega)$, $\BV(\Omega)$ is compactly embedded in $L^1(\Omega)$, and $TV(\cdot)$ is convex and lower-semicontinuous in $L^1(\Omega)$-norm. Thus in this setting, assumptions  1 and 2 are satisfied.

\begin{remark}\label{SR}
    The regularization framework devised above simplifies if we assume that the set $A$ belongs to a Hilbert space $H$ that is compactly embedded in the space $B$. An example for this setting is $B=L^{\infty}(\Omega)$ and $H=H^2(\Omega)$, for a suitable domain $\Omega$.
\end{remark}

\bigskip

Our objective is to approximate \eqref{MinOLS} by the following family of regularized optimization problems: For $n\in \mathbb{N}$, find $a_{\varsigma_n}\in A$ by solving
\begin{equation}\label{RMinOLS} 
    \min_{a\in A}\widehat{J}_{\kappa_n}(a):=\frac12\|u_{\varsigma_n}(a)-z_{\delta_n}\|^2_Z+\kappa_n R(a),
\end{equation}
where $u_{\varsigma_n}(a)$ is the unique solution of \eqref{RVP}, $\kappa_n>0$, and $R$ is the regularizer defined above.

\bigskip

The following main result of this section shows that \eqref{RMinOLS} approximates \eqref{MinOLS}:
\begin{theorem}\label{ExisConT-ROLS}
    Assume that the following conditions hold:
    \begin{enumerate}
        \item The set  $A$ is bounded in $\widehat{B}$, the image set $\mathcal{U}(A)$ is a  bounded set, and for each $a\in A$,  the solution set $\mathcal{U}(a)$ is nonempty.
        \item For $a\in A$, either $\mathcal{U}(a)$ is a singleton, or $Z=V$, $\ell_{\delta_n}(v)=\langle z_{\delta_n},v\rangle_Z$ and $S(u,v)=\langle u,v\rangle_Z$.
    \end{enumerate}
    Then optimization problem \eqref{MinOLS} has a solution, and for each $n\in \mathbb{N}$, optimization problem \eqref{RMinOLS} has a solution $a_{\varsigma_n}$. Moreover, there is a subsequence $\{a_{\varsigma_n}\}$ converging in $\|\cdot\|_L$ to a solution of \eqref{MinOLS}. Finally, for any solution $a_{\varsigma_n}$ of \eqref{MinOLS}, there is  a unique $p_{\varsigma_n}\in V$ such that
    \begin{align}
        T_{\tau_n}(a_{\varsigma_n},p_{\varsigma_n},v)+
        \epsilon_n S(p_{\varsigma_n},v)&=\langle z-u_{\varsigma_n}(a_{\varsigma_n}),v\rangle_Z,\quad \text{for every}\ v\in V,\label{OCa}\\
        T_{\tau_n}(a-a_{\varsigma_n},u_{\varsigma_n}(a_{\varsigma_n}),p_{\varsigma_n})&\geq \kappa_n (R(a_{\varsigma_n})-R(a)),\quad
        \text{for every } a\in A.\label{OCb}
    \end{align}
\end{theorem}
\begin{proof}
    We begin by showing that \eqref{MinOLS} has a solution. Since $\mathcal{U}(a)$ is nonempty for each $a\in A$, optimization problem \eqref{MinOLS} is well-defined. Moreover, since for each $a\in A$,  $\widehat{J}(a)$ is bounded from below, there is a minimizing sequence $\{a_n\}$ in $A $ such that $ \lim_{n\to \infty}\widehat{J}(a_n)=\inf\{\widehat{J}(a),\ a\in A  \}$.
    Since $A$ is bounded, the sequence $\{a_n\}$ is bounded in $\widehat{B}$, and due to the compact embedding of $\widehat{B }$ into $L$, it has a subsequence which converges strongly in $\|\cdot\|_{L}$.  Keeping the same notation for subsequences as well, let $\{a_n\}$ be the subsequence converging in $\|\cdot\|_L$ to some $\bar{a}\in A$.
    Let $u_n\in \mathcal{U}(a_n)$ be arbitrary. Since $\mathcal{U}(A)$ is bounded, $\{u_n\}$ is bounded, and hence contains a weakly convergent subsequence. Let $\{u_n\}$ be the subsequence which converges to some $\bar{u}\in V$. We claim that $\bar{u}\in \mathcal{U}(\bar{a})$. By the definition of $(a_n,u_n)$, we have
    \begin{equation*}
        T(a_n,u_n,v)=m(v),\quad \text{for every}\ v\in V,
    \end{equation*}
    which can be rearranged as follows
    \begin{equation*}
        T(a_n-\bar{a},u_n,v)+T(\bar{a},u_n-\bar{u},v)+T(\bar{a},\bar{u},v)=m(v),\quad \text{for every}\ v\in V,
    \end{equation*}
    and by passing this equation to the limit $n\to\infty$, we obtain
    \begin{equation*}
        T(\bar{a},\bar{u},v)=m(v),\quad \text{for every}\ v\in V,
    \end{equation*}
    which means that $\bar{u}\in \mathcal{U}(\bar{a})$. The optimality of $\bar{a}$ now is direct consequence of the weak-lower-semicontinuity of $\|\cdot\|_Z$ and the lower-semicontinuity of $R$.

    We now return to \eqref{RMinOLS}. Evidently, for a fixed $n\in \mathbb{N}$, the existence of a solution $a_{\varsigma_n}$ for  \eqref{RMinOLS}  is a consequence of the arguments just used. Indeed, for any fixed $n\in \mathbb{N}$, \eqref{RVP} is uniquely solvable and the solution is bounded because of the ellipticity of $T_{\tau_n}+\epsilon_n S$.

    For simplicity, we set $a_n:=a_{\varsigma_n}$. Since $A$ is bounded, the sequence of solutions $\{a_n\}$ is uniformly bounded in $\widehat{B}$. As before, let $\{a_n\}$ be a subsequence that
    converges strongly to some $\bar{a}\in A$ in $L$.  Let $\{u_n\}$, where $u_n:=u_{\varsigma_n}(a_n)$, be the corresponding sequence of the solutions of the regularized variational problem \eqref{RVP}. That is, we have
    \begin{equation*}
        T_{\tau_n}(a_n,u_n,v)+\epsilon_n S(u_n,v) =m_{\nu_n}(v)+\epsilon_n\ell_{\delta_n}(v),\ \ \text{for every}\ v\in V.
    \end{equation*}
    We shall prove that $\{u_n\}$ is a bounded sequence. By assumption, for every $a\in A$, the solution set $\mathcal{U}(a)$ is nonempty. Let $\tilde{u}_n\in \mathcal{U}(a_n)$ be chosen arbitrarily. Since $\mathcal{U}(A)$ is bounded by assumption, the sequence $\{\tilde{u}_n\}$ is bounded. Moreover, we have
    \begin{equation*}
        T(a_n,\tilde{u}_n,v)=m(v),\quad \text{for every}\ v\in V. 
    \end{equation*}
    We combine the above two variational problems and rearrange them to obtain
    \begin{equation*}
        T_{\tau_n}(a_n,\tilde{u}_n,v)-T(a_n,\tilde{u}_n,v)+\epsilon_n S(u_n,v)+m(v)-m_{\nu_n}(v)-\epsilon_n\ell_{\delta_n}(v)-T_{\tau_n}(a_n,\tilde{u}_n-u_n,v)=0.
    \end{equation*}
    Setting $v=\tilde{u}_n-u_n$ and using the fact that $T_{\tau_n}(a_n,\tilde{u}_n-u_n,\tilde{u}_n-u_n)\geq 0$, we obtain
    \begin{equation*}
        \begin{aligned}
            \epsilon_nS( u_n,u_n) &\leq \epsilon_n S(u_n,\tilde{u}_n)+ T_{\tau_n}(a_n,\tilde{u}_n,\tilde{u}_n-u_n)-T(a_n,\tilde{u}_n,\tilde{u}_n-u_n)-\epsilon_n\ell_{\delta_n}(\tilde{u}_n-u_n).\\
            \MoveEqLeft[-1]+m(\tilde{u}_n-u_n)-m_{\nu_n}(\tilde{u}_n-u_n)\\
            &\leq \epsilon_n\beta_0\|u_n\|_V\|\tilde{u}_n\|_V+\tau_n\|a_n\|_B\|\tilde{u}_n\|_V\|\tilde{u}_n-u_n\|_V
            +\epsilon_n\|\ell_{\delta_n}\|_{V^*}\|\tilde{u}_n-u_n\|_V \\
            \MoveEqLeft[-1]+\nu_n\|\tilde{u}_n-u_n\|_V,
        \end{aligned}
    \end{equation*}
    implying
    \begin{equation*}
        \|u_n\|_V\leq \frac{\beta_0}{\alpha_0}\|\tilde{u}_n\|_V+\left[\frac{\tau_n}{\alpha_0\epsilon_n}\|a_n\|_B\|\tilde{u}_n\|_V
            +\frac{\nu_n}{\alpha_0\epsilon_n}
        +\frac{\delta_n+\|\ell\|_{V^*}}{\alpha_0}\right]\left[\frac{\|\tilde{u}_n\|_V}{\|u_n\|_V}+1\right],
    \end{equation*}
    which confirms the boundedness of $\{u_n\}$.

    The reflexivity of $V$ ensures that $\{u_n\}$ has a weakly convergent subsequence. Keeping the same notation for subsequences, let $\{u_n\}$ be a subsequence converging weakly to some $\bar{u}$. We shall show that $\bar{u}\in \mathcal{U}(\bar{a})$. Since  $a_n$ is a minimizer of \eqref{RMinOLS},
    we have
    \begin{equation*}
        T_{\tau_n}(a_n,u_n,v)+\epsilon_n S(u_n,v)=m_{\nu_n}(v)+\epsilon_n\ell_{\delta_n}(v)\quad \text{for every}\ v\in V,
    \end{equation*}
    and by using the rearrangement
    \begin{equation*}
        \begin{aligned}
            T_{\tau_n}(a_n,u_n,v)&=T(a_n,u_n,v)+T_{\tau_n}(a_n,u_n,v)-T(a_n,u_n,v)\\
                                 &=T(a_n-\bar{a},u_n,v)+T(\bar{a},u_n-\bar{u},v)+T(\bar{a},\bar{u},v)\\
                                 &+T_{\tau_n}(a_n,u_n,v)-T(a_n,u_n,v),
        \end{aligned}
    \end{equation*}
    we obtain the following equation
    \begin{multline*}
        T(a_n-\bar{a},u_n,v)+T(\bar{a},u_n-\bar{u},v)+T(\bar{a},\bar{u},v)+T_{\tau_n}(a_n,u_n,v)-T(a_n,u_n,v)\\
        +\epsilon_n S(u_n,v)=m_{\nu_n}(v)+\epsilon_n\ell_{\delta_n}(v),
    \end{multline*}
    which due to the imposed conditions, when passed to the limit $n\to \infty$,  implies that
    \begin{equation*}
        T(\bar{a},\bar{u},v)=m(v),
    \end{equation*}
    and because of the fact that $v\in V$ was chosen arbitrary, confirms that $\bar{u}\in \mathcal{U}(\bar{a})$.

    The optimality of $a_n\in A$ for \eqref{RMinOLS} means that for $n\in \mathbb{N}$ and each $a\in A$,  we have
    \begin{equation}\label{UseNew}
        \widehat{J}_{\kappa_n}(a_n):=\frac12\|u_{\varsigma_n}(a_n)-z_{\delta_n}\|^2_Z+\kappa_n R(a_n)\leq \frac{1}{2}\|u_{\varsigma_n}(a)-z_{\delta_n}\|^2_Z+\kappa_nR(a),
    \end{equation}
    where $u_{\varsigma_n}(a)$ is the solution of regularized optimization problem \eqref{RVP}.

    Let $(\hat{a},\hat{u})$ be a solution of \eqref{MinOLS}. Before any further advancement, we first analyze the behavior of $u_{\varsigma_n}(\hat{a})$. By the definition of $u_{\varsigma_n}(\hat{a})$, we have
    \begin{equation}\label{PVI1}
        T_{\tau_n}(\hat{a},u_{\varsigma_n}(\hat{a}),v)+\epsilon_n S(u_{\varsigma_n}(\hat{a}),v)=m_{\nu_n}(v)+\epsilon_n\ell_{\delta_n}(v),\quad \text{for every}\ v\in V.
    \end{equation}
    As in earlier part of this proof, it can be shown that $\{u_{\varsigma_n}(\hat{a})\}$ is uniformly bounded. Therefore, there is a subsequence $\{u_{\varsigma_n}(\hat{a})\}$ converging weakly to some $\bar{u}(\hat{a})\in \mathcal{U}(\hat{a})$.

    Recalling that the set $\mathcal{U}(\hat{a})$ is closed and convex, we consider the following variational inequality: Find $\tilde{u}(\hat{a})\in \mathcal{U}(\hat{a})$ such that
    \begin{equation}\label{PVI10}
        S(\tilde{u}(\hat{a}),w-\tilde{u}(\hat{a}))\geq \ell(w-\tilde{u}(\hat{a})),\ \text{for every}\ w\in \mathcal{U}(\hat{a}).
    \end{equation}
    Due to the ellipticity of $S(\cdot,\cdot)$, the above variational inequality has a unique solution $\tilde{u}(\hat{a})$. Furthermore,
    since $\tilde{u}(\hat{a})\in \mathcal{U}(\hat{a})$, we have
    \begin{equation}\label{PVI2}
        T(\hat{a},\tilde{u}(\hat{a}),v)=m(v),\quad \text{for every}\ v\in V.
    \end{equation}
    We combine \eqref{PVI1} and \eqref{PVI2} to obtain
    \begin{multline*}
        T(\hat{a},u_{\varsigma_n}(\hat{a})-\tilde{u}(\hat{a}),v)+T_{\tau_n}(\hat{a},u_{\varsigma_n}(\hat{a}),v)-T(\hat{a},u_{\varsigma_n}(\hat{a}),v)+
        \epsilon_nS(u_{\varsigma_n}(\hat{a}),v)\\
        =m_{\nu_n}(v)-m(v)+\epsilon_n\ell_{\delta_n}(v)
    \end{multline*}
    and by setting $v=\tilde{u}(\hat{a})-u_{\varsigma_n}(\hat{a})$, and using the positivity of $T$, we get
    \begin{multline}
        \left[\frac{\tau_n}{\epsilon_n}\|\hat{a}\|_B\|u_{\varsigma_n}(\hat{a})\|_V +\frac{\nu_n}{\epsilon_n}+\delta_n\right]\|\tilde{u}(\hat{a})-u_{\varsigma_n}(\hat{a})\|_V -\ell(\tilde{u}(\hat{a})-u_{\varsigma_n}(\hat{a}))\\
        \geq S(u_{\varsigma_n}(\hat{a}),u_{\varsigma_n}(\hat{a})-\tilde{u}(\hat{a}))\geq S(\tilde{u}(\hat{a}),u_{\varsigma_n}(\hat{a})-\tilde{u}(\hat{a}))\label{PVI20}
    \end{multline}
    Since the bilinear form $S$ is positive, we have 
    \begin{equation*}
        S(\bar{u}(\hat{a}),\bar{u}(\hat{a}))\leq \liminf_{n\to \infty}S(u_{\varsigma_n}(\hat{a}),u_{\varsigma_n}(\hat{a})),
    \end{equation*}
    which, by taking \eqref{PVI20} into account, implies that
    \begin{equation}\label{MVI}
        S(\bar{u}(\hat{a}),\tilde{u}(\hat{a})-\bar{u}(\hat{a})) \geq \ell(\tilde{u}(\hat{a})-\bar{u}(\hat{a})).
    \end{equation}
    We set $w=\bar{u}(\hat{a})$ in \eqref{PVI10} to obtain
    \begin{equation*}
        S(\tilde{u}(\hat{a}),\bar{u}(\hat{a})-\tilde{u}(\hat{a}))\geq \ell (\bar{u}(\hat{a})-\tilde{u}(\hat{a})),
    \end{equation*}
    which when combined with \eqref{MVI} yields $S(\bar{u}(\hat{a})-\tilde{u}(\hat{a}),\tilde{u}(\hat{a})-\bar{u}(\hat{a}))\geq 0$, implying
    \begin{equation*}
        0\geq S(\bar{u}(\hat{a})-\tilde{u}(\hat{a}),\bar{u}(\hat{a})-\tilde{u}(\hat{a}))\geq \alpha_0\|\bar{u}(\hat{a})-\tilde{u}(\hat{a})\|^2_V,
    \end{equation*}
    and hence $\bar{u}(\hat{a})=\tilde{u}(\hat{a})$. Since $\bar{u}(\hat{a})$ is unique, the whole sequence $u_{\varsigma_n}(\hat{a})$ converges weakly to $\bar{u}(\hat{a})$. The convergence is in  fact strong due to \eqref{PVI20}. Indeed, by the coercivity of $S$, we have
    \begin{equation*}
        \alpha_0\|u_{\varsigma_n}(\hat{a})-\bar{u}(\hat{a})\|^2_V\leq S(u_{\varsigma_n}(\hat{a}),u_{\varsigma_n}(\hat{a})-\bar{u}(\hat{a}))-S(\bar{u}(\hat{a}),u_{\varsigma_n}(\hat{a})-\bar{u}(\hat{a})),
    \end{equation*}
    where $S(u_{\varsigma_n}(\hat{a}),u_{\varsigma_n}(\hat{a})-\bar{u}(\hat{a}))\to 0$ as $n\to \infty$ by using \eqref{PVI20} and  $S(\bar{u}(\hat{a}),u_{\varsigma_n}(\hat{a})-\bar{u}(\hat{a}))\to 0$ as $n\to \infty$ by the linearity of $S$. Hence the strong convergence of $\{u_{\varsigma_n}(\hat{a})\}$ to $\bar{u}(\hat{a})$ follows.

    The above observations are valid when $\mathcal{U}$ is a set-valued map. We now prove the final assertion by assuming that $Z=V$, for any $v\in V$, we have $\ell_{\delta_n}(v)=\langle z_{\delta_n},v\rangle_V$ and $S(u,v)=\langle u,v\rangle_V$. Then, it follows from \eqref{PVI10} that for an arbitrary
    $\breve{u}(\hat{a})\in \mathcal{U}(\hat{a})$, we have
    \begin{equation*}
        \langle \bar{u}(\hat{a})-z,\breve{u}-\bar{u}(\hat{a})\rangle_V\geq 0 
    \end{equation*}
    which implies that
    \begin{equation*}
        \|\bar{u}(\hat{a})-z\|_Z\leq \|\breve{u}(\hat{a})-z\|_Z,
    \end{equation*}
    and hence  $\bar{u}(\hat{a})$ is the closest element to $z$ among all the elements $\breve{u}(\hat{a})\in \mathcal{U}(\hat{a})$.

    Therefore, as before, we have
    \begin{equation*}
        \begin{aligned}
            \|u(\bar{a})-z\|^2_Z&\leq \liminf_{n\to \infty}\left\{\|u_{\varsigma_n}(\hat{a})-z_{\delta_n}\|^2_Z+\kappa_n R(\hat{a})\right\},\\
                                &\leq \limsup_{n\to \infty}\|u_{\varsigma_n}(\hat{a})-z\|^2_Z\\
                                &=\|\bar{u}(\hat{a})-z\|^2_Z\\
                                &\leq \|\breve{u}(\hat{a})-z\|^2_Z,
        \end{aligned}
    \end{equation*}
    where $\breve{u}(\hat{a})\in \mathcal{U}(\hat{a})$ is arbitrary. In other words, the above inequality confirms the existence of an element $(\bar{a},u(\bar{a}))\in \text{graph}(\mathcal{U})$ such that  for every $(a,u)\in \text{graph}(\mathcal{U})$, we have
    \begin{equation*}
        \|u(\bar{a})-z\|^2_Z \leq \|u-z\|^2_Z 
    \end{equation*}
    and hence $\bar{a}\in A$ is a minimizer of \eqref{MinOLS}. Evidently, if $\mathcal{U}(a)$ is singleton for each $a\in A$, then the supplied arguments remain valid for any $S$ and $\ell$.

    Finally, we proceed to prove \eqref{OCa} and \eqref{OCb}.  Note that a necessary optimality condition for $a_{\varsigma_n}$ to be a solution of \eqref{RMinOLS} is the following variational inequality
    \begin{equation}\label{NOC-VI}
        D\widehat{J}_{\kappa_n}(a_{\varsigma_n})(a-a_{\varsigma_n})\geq \kappa_n (R(a_{\varsigma_n})-R(a)),\quad
        \text{for every}\, a\in A,
    \end{equation}
    where $\widehat{J}_{\kappa_n}(a_{\varsigma_n}):=\frac{1}{2}\Vert u_{\varsigma_n}(a_{\varsigma_n})-z_{\delta_n}\Vert_{Z}^{2}$ and
    $D\widehat{J}_{\kappa_n}(a)(b)=\langle Du_{\varsigma_n}(a)(b),u_{\varsigma_n}(a)-z_{\delta_n}\rangle_Z$.

    For $n\in \mathbb{N}$, we define the adjoint equation: Find
    $p_{\varsigma_n}\in V$, such that
    \begin{equation}\label{AE-OLS}
        T_{\tau_n}(a_{\varsigma_n},p_{\varsigma_n},v)+\epsilon_nS(p_{\varsigma_n},v)=\left\langle z_{\delta_n}-u_{\varsigma_n}(a_{\varsigma_n}),v\right\rangle_{Z},\quad \text{for every}\ v\in V.
    \end{equation}
    Evidently, \eqref{AE-OLS} has a unique solution $p_{\varsigma_n}$. Taking $v=Du_{\varsigma_n}(a_{\varsigma_n})(a-a_{\varsigma_n})$, we get
    \begin{equation*}
        \begin{aligned}
            \left\langle Du_{\varsigma_n}(a_{\varsigma_n})(a-a_{\varsigma_n}),u_{\varsigma_n}(a_{\varsigma_n})-z_{\delta_n}\right\rangle_Z
            &=-T_{\tau_n}(a_{\varsigma_n},p_{\varsigma_n},Du_{\varsigma_n}(a_{\varsigma_n})(a-a_{\varsigma_n}))\\
            \MoveEqLeft[-1]-
            \epsilon_nS(p_{\varsigma_n},Du_{\varsigma_n}(a_{\varsigma_n})(a-a_{\varsigma_n}))\\
            &=-T_{\tau_n}(a_{\varsigma_n},Du_{\varsigma_n}(a_{\varsigma_n})(a-a_{\varsigma_n}),p_{\varsigma_n})\\
            \MoveEqLeft[-1]-\epsilon_nS(Du_{\varsigma_n}(a_{\varsigma_n})(a-a_{\varsigma_n}),p_{\varsigma_n})\\
            &=T_{\tau_n}(a-a_{\varsigma_n},u_{\varsigma_n}(a_{\varsigma_n}), p_{\varsigma_n}),
        \end{aligned}
    \end{equation*}
    by \eqref{var4RP} and \eqref{OCb} follows by using the above expression in \eqref{NOC-VI}. The proof is complete.
\end{proof}
\begin{remark} Since $(\hat{a},\hat{u})$ is a minimizer of \eqref{MinOLS}, we have $\|\hat{u}(\hat{a})-z\|_Z\leq \|u-z\|_Z$, for each $(a,u)\in \text{graph}(\mathcal{U})$. Since $\|\bar{u}(\hat{a})-z\|_Z\geq \|\hat{u}(\hat{a})-z\|_Z$ is possible, we can't use $(\hat{a},\hat{u})$ to show that $(\bar{a},\bar{u})$ is optimal.  We circumvented this difficulty by showing  $\|\bar{u}(\hat{a})-z\|_Z\leq \|\hat{u}(\hat{a})-z\|_Z$. A practical implication of the condition $Z=V$ is that typically more regular data is required.
\end{remark}
\begin{remark}
    For $Z=V$, $\ell_{\delta_n}(\cdot)=\langle z_{\delta_n},\cdot\rangle_V$ and
    $S(\cdot,\cdot)=\langle \cdot,\cdot \rangle_V$, \eqref{RVP} reduces to
    \begin{equation*}
        T_{\tau_n}(a,u_{\varsigma_n}(a),v)+\epsilon_n \langle u_{\varsigma_n}(a)-z_{\delta_n},v\rangle_V=m_{\nu_n}(v),\ \ \text{for every}\ v\in V,
    \end{equation*}
    which steers the regularized solutions towards the solution of \eqref{VP} that is closest to $z$. If $\ell_{\delta_n}(\cdot)=0$, then the regularized solutions converge to a minimum norm solution of \eqref{VP}. We also note that if the sequence of adjoint solutions $\{p_{\varsigma_n}\}$ is bounded, then by passing \eqref{OCa} and \eqref{OCb} to limit, we shall derive optimality conditions for \eqref{MinOLS}. Because an optimality condition for \eqref{MinOLS} would involve the derivative of the set-valued parameter-to-solution map, such convergence result could shed some light on its contingent differentiability.
\end{remark}

\subsection{The Modified OLS Approach}\label{sec:framework:mols}

We shall now focus on the following MOLS-based constrained optimization problem
\begin{equation}\label{MinMOLS} 
    \min_{a\in A}J(a):=\frac12T(a,u(a)-z,u(a)-z),
\end{equation}
which aims to minimize the energy associated to underlying noncoercive variational problem \eqref{VP}. Here $u(a)\in \mathcal{U}(a)$ and $z\in V$ is the measured data. Studies related to the MOLS functional and its extensions can be found in \cite{GocJadKha06,GocKha05,GocKha09,JadKhaRusSamWin14}.

We continue to assume that $\{\epsilon_n\}$, $\{\tau_n\}$, $\{\kappa_n\}$, $\{\delta_n\}$, and $\{\nu_n\}$ are sequence of  positive reals, $\ell\in V^*$, and for each $n\in \mathbb{N}$, $m_{\nu_n}\in V^*$, $\ell_{\delta_n}\in V^*$, and  $z_{\delta_n}\in  V$ satisfying \eqref{Noise}. Furthermore, the trilinear form $T_{\tau_n}:B\times V\times V\to \mathbb{R}$ satisfies \eqref{TNoise} and the bilinear and symmetric form  $S:V\times V\to \mathbb{R}$ satisfies \eqref{SCond}.

We again consider the regularized problem: Given $a\in A$, find $u_{\varsigma_n}(a)\in V$
such that
\begin{equation}\label{RVP1}
    T_{\tau_n}(a,u_{\varsigma_n}(a),v)+\epsilon_n S(u_{\varsigma_n}(a),v)=m_{\nu_n}(v)+\epsilon_n\ell_{\delta_n}(v),\ \ \text{for every}\ v\in V,
\end{equation}
where $\epsilon_n>0$ is a regularization parameter and $\varsigma_n:=(\epsilon_n,\tau_n,\nu_n,\delta_n)$.   For a fixed $n\in \mathbb{N}$, let $u_{\varsigma_n}(a)$ be the unique solution of  \eqref{RVP1}.

We first consider the following analogue of the MOLS objective with perturbed data:
\begin{equation}\label{MOLSNC} J_{\epsilon_n}(a):=\frac12T_{\tau_n}(a,u_{\varsigma_n}(a)-z_{\delta_n},u_{\varsigma_n}(a)-z_{\delta_n})+\frac{\epsilon_n}{2}S(u_{\varsigma_n}(a)-z_{\delta_n},u_{\varsigma_n}(a)-z_{\delta_n}).
\end{equation}

We have the following result:
\begin{theorem}\label{THR1}
    For each $n\in \mathbb{N}$, the modified output least-squares functional \eqref{MOLSNC} is convex in  $A$.
\end{theorem}
\begin{proof} For each $n\in \mathbb{N}$, the functional $J_{\epsilon_n}$ is evidently infinitely differentiable.  The first
    derivative is derived by the using the chain rule:
    \begin{multline*}
        DJ_{\epsilon_n}(a)\delta a=\frac{1}{2}T_{\tau_n}(\delta a,u_{\varsigma_n}(a)-z_{\delta_n},u_{\varsigma_n}(a)-z_{\delta_n})+T_{\tau_n}(a,Du_{\varsigma_n}(a)\delta a,u_{\varsigma_n}(a)-z_{\delta_n})\\
        +\epsilon_n S(Du_{\varsigma_n}(a)\delta a, u_{\varsigma_n}-z_{\delta_n}).
    \end{multline*}
    By using \eqref{var4RP}, we have
    \begin{equation*}
        T_{\tau_n}(a,Du_{\varsigma_n}(a)\delta a,u_{\varsigma_n}(a)-z_{\delta_n})+\epsilon_n S(Du_{\varsigma_n}(a)\delta a,u_{\varsigma_n}(a)-z_{\delta_n}) =-T_{\tau_n}(\delta a,u_{\varsigma_n},u_{\varsigma_n}(a)-z_{\delta_n}),
    \end{equation*}
    and hence
    \begin{equation}\label{MOLS-DER}
        \begin{aligned}[b]
            DJ_{\epsilon_n}(a)\delta a&=\frac{1}{2}T_{\tau_n}(\delta a,u_{\varsigma_n}(a)-z_{\delta_n},u_{\varsigma_n}(a)-z_{\delta_n})-
            T_{\tau_n}(\delta a,u_{\varsigma_n}(a),u_{\varsigma_n}(a)-z_{\delta_n})\\
            &=-\frac{1}{2}T_{\tau_n}(\delta a,u_{\varsigma_n}(a)+z_{\delta_n},u_{\varsigma_n}(a)-z_{\delta_n}).
        \end{aligned}
    \end{equation}
    It now follows that
    \begin{equation*}
        \begin{aligned}
            D^2J_{\epsilon_n}(a)(\delta a,\delta a)&=-\frac{1}{2}T_{\tau_n}(\delta a,Du_{\varsigma_n}(a)\delta a,u_{\varsigma_n}(a)-z_{\delta_n})-
            \frac{1}{2}T_{\tau_n}(\delta a,u_{\varsigma_n}(a)+z_{\delta_n},Du_{\varsigma_n}(a)\delta a)\\
            &=-T_{\tau_n}(\delta a,u_{\varsigma_n}(a),Du_{\varsigma_n}(a)\delta a )\\
            &=T_{\tau_n}(a,Du_{\varsigma_n}(a)\delta a,Du_{\varsigma_n}(a)\delta a)+\epsilon_n S(Du_{\varsigma_n}(a)\delta a,Du_{\varsigma_n}(a)\delta a),
        \end{aligned}
    \end{equation*}
    where in the last step, we used
    \begin{equation*}
        T_{\tau_n}(a,Du_{\varsigma_n}(a)\delta a,Du_{\varsigma_n}(a)\delta a)+\epsilon_n S(Du_{\varsigma_n}(a)\delta a,Du_{\varsigma_n}(a)\delta a) =-T_{\tau_n}(\delta a,u_{\varsigma_n},Du_{\varsigma_n}(a)\delta a),
    \end{equation*}
    which follows from \eqref{var4RP}.  We notice, in particular,
    that the following inequality holds for all $a$ in the interior of $A$:
    \begin{equation}\label{convex}
        D^2J_{\epsilon_n}(a)(\delta a,\delta a)\geq\epsilon_n\beta_0\|Du_{\varsigma_n}(a)\delta a\|_V^2.
    \end{equation}
    Thus $J_{\epsilon_n}$ is a smooth and convex functional.
\end{proof}

Our objective is to approximate \eqref{MinMOLS} by the following family of regularized MOLS based optimization problems: For $n\in \mathbb{N}$, find $a_{\varsigma_n}\in A$ by solving
\begin{multline}  \label{RMinMOLS}
    \min_{a\in A}J_{\kappa_n}(a):=\frac12T_{\tau_n}(a,u_{\varsigma_n}(a)-z_{\delta_n},u_{\varsigma_n}(a)-z_{\delta_n})
    +\frac{\epsilon_n}{2}S(u_{\varsigma_n}(a)-z_{\delta_n},u_{\varsigma_n}(a)-z_{\delta_n})
    +\kappa_n R(a),
\end{multline}
where $u_{\varsigma_n}(a)$ is the unique solution of the regularized variational problem \eqref{RVP1}.

We have the following result:
\begin{theorem}\label{ExisConT-MOLS} Assume that the following conditions hold:
    \begin{enumerate}
        \item  The set  $A$ is bounded in $\widehat{B}$, for each $a\in A$,  the solution set $\mathcal{U}(a)$ is nonempty, and the image set  $\mathcal{U}(A)$ is a  bounded set.
        \item For each $a\in A$, either $\mathcal{U}(a)$ is a singleton, or $\ell_{\delta_n}(v)=\langle z_{\delta_n},v\rangle_V$ and $S(u,v)=\langle u,v\rangle_V$.
        \item For every $a\in A$, and any $u,w\in V$, $\|u-z\|_V\leq \|w-z\|_V$ implies $T(a,u-z,u-z)\leq T(a,w-z,w-z)$.
    \end{enumerate}
    Then, the optimization problem \eqref{MinMOLS} has a solution, and for each $n\in \mathbb{N}$, optimization problem \eqref{RMinMOLS} has a solution. Moreover, there is a subsequence $\{a_{\varsigma_n}\}\subset A$ converging in $\|\cdot\|_L$ to a solution of \eqref{MinMOLS}.  Furthermore, a necessary and sufficient optimality condition for any solution $a_{\varsigma_n}$ of \eqref{MinMOLS} is the following variational inequality:
    \begin{equation}\label{MOLS-OC}
        -\frac{1}{2}T_{\tau_n}(a-a_{\varsigma_n},u_{\varsigma_n}(a_{\varsigma_n})+z_{\delta_n},u_{\varsigma_n}(a_{\varsigma_n})-z_{\delta_n})\geq \kappa_n[R(a_{\varsigma_n})-R(a)],\quad \text{for every } a\in A.
    \end{equation}
    Finally, \eqref{MOLS-OC}, when passed to the limit $n\to \infty$, results in the following variational inequality
    \begin{equation}\label{MOLS-OCC}
        -\frac{1}{2}T(a-\bar{a},u(\bar{a})+z,u(\bar{a})-z)\geq 0,\quad \text{for every}\ a\in A.
    \end{equation}
\end{theorem}
\begin{proof} It follows by standard arguments that for each $n\in \mathbb{N}$, \eqref{RMinMOLS}  has a solution. By assumption $\{a_n\}$ is a bounded sequence, and due to the compact imbedding of $\widehat{B}$ into $L$, it possesses a strongly convergent subsequence. Let $\{a_n\}$ be the  subsequence  which converges strongly to some $\bar{a}\in A$. Let $\{u_n\}$ be the corresponding sequence of the solutions of the regularized variational problems. As in the proof of \cref{ExisConT-ROLS}, we can show that $\{u_n\}$ is bounded, and there is a subsequence $\{u_n\}$ converging weakly to some $\bar{u}=u(\bar{a})$.

    The only new step is to show that for $a_n\to \bar{a}$ and $u_n:=u_{\varsigma_n}(a_n)\rha u(\bar{a})$, we have
    \begin{equation}\label{Func}
        T_{\tau_n}(a_n,u_n-z_{\delta_n},u_n-z_{\delta_n})\to T(a,u(\bar{a})-z,u(\bar{a})-z).
    \end{equation}
    Indeed, to prove this convergence, we note that for every $v\in V$, we have
    \begin{equation*}
        T_{\tau_n}(a_n,u_{n},v)+\epsilon_n S(u_n,v)=m_{\nu_n}(v)+\epsilon_n\ell_{\delta_n}(v),
    \end{equation*}
    which, for the choice $v=u_n-z_{\delta_n}$, can be rearranged as follows
    \begin{equation*}
        \begin{aligned}
            T(a_n,u_n-z_{\delta_n},u_n-z_{\delta_n})&=-T(a_n,z_{\delta_n},u_n-z_{\delta_n})+m(u_n-z_{\delta_n})-T_{\tau_n}(a_n,u_n,u_n-z_{\delta_n})\\
            \MoveEqLeft[-1]+T(a_n,u_n,u_n-z_{\delta_n})-\epsilon_nS(u_n,u_n-z_{\delta_n})+\epsilon_n \ell_{\delta_n}(u_n-z_{\delta_n})\\
            \MoveEqLeft[-1]+m_{\nu_n}(u_n-z_{\delta_n})-m(u_n-z_{\delta_n}),
        \end{aligned}
    \end{equation*}
    and since the right-hand side of the above equation converges to $T(\bar{a},-z,\bar{u}-z)+m(\bar{u}-z)$, which, due to the fact $\bar{u}\in\mathcal{U}(\bar{a})$, equals to $T(\bar{a},\bar{u}-z,\bar{u}-z)$, the desired convergence follows.

    Let $(\hat{a},\hat{u})$ be a solution of \eqref{MinMOLS}.  Then, by using \eqref{Func}, we have
    \begin{equation}\label{MinE10}
        \begin{aligned}[b]
            T(\bar{a},u(\bar{a})-z,u(\bar{a})-z)&= \lim_{n\to \infty}T_{\tau_n}(a_n,u_n(a_n)-z_{\delta_n},u_n(a_n)-z_{\delta_n})\\
                                                &\leq \liminf_{n\to \infty}\left\{T_{\tau_n}(a_n,u_n(a_n)-z_{\delta_n},u_n(a_n)-z_{\delta_n})+\kappa_n R(a_n)\right\}\\
                                                &\leq \liminf_{n\to \infty}\left\{T_{\tau_n}(\hat{a},u_n(\hat{a})-z_{\delta_n},u_n(\hat{a})-z_{\delta_n})+\kappa_n R(\hat{a})\right\}\\
                                                &\leq \limsup_{n\to \infty}T(\hat{a},u_n(\hat{a})-z,u_n(\hat{a})-z)\\
                                                &=T(\hat{a},u(\hat{a})-z,u(\hat{a})-z),
        \end{aligned}
    \end{equation}
    which, as for the case of the OLS objective ensures that $\bar{a}$ is a solution of \eqref{MinE10} and the proof is complete. Note that when $\mathcal{U}(a)$, for $a\in A$ is not singleton, we need additionally condition (3) on the trilinear form.

    Due to the convexity of the MOLS functional, a necessary and sufficient optimality condition for $a_{\varsigma_n}$ to be a solution of \eqref{RMinMOLS} is the variational inequality of second-kind
    \begin{equation}\label{Mr
        NOC-VI}
        DJ_{\kappa_n}(a_{\varsigma_n})(a-a_{\varsigma_n})\geq \kappa_n (R(a_{\varsigma_n})-R(a)),\quad
        \text{for every}\, a\in A,
    \end{equation}
    where $J_{\kappa_n}$ is defined in \eqref{RMinMOLS}. Condition \eqref{MOLS-OC} is then follows from the derivative characterization \eqref{MOLS-DER}. Variational inequality \eqref{MOLS-OCC} is a consequence of the properties of $T$ and the facts that  $a_{\varsigma_n}\to \bar{a}$ in $\|\cdot\|_L$, $u_{\varsigma_n}\to \bar{u}:=u(\bar{a})$, and $z_{\delta_n}\to z$. The proof is complete.
\end{proof}

\section{First-order and Second-order Adjoint Approach for OLS}\label{sec:adjoint}

We now describe the first-order and the second-order adjoint approaches to compute the first-order and the second-order derivatives of the regularized OLS functional. These formulae can be discretized to derive an efficient scheme for the computation of the gradient and the Hessian of regularized OLS objective. The gradient computation by the adjoint approach avoids a direct computation of the first-order derivative of the regularized parameter-to-solution map whereas the Hessian computation by the second-order adjoint approach avoids a direct computation of the second-order  derivative of the regularized parameter-to-solution map. Adjoint methods have been used extensively in the literature and some of the recent developments can be found in \cite{ChoJadKahKhaSam17,JadKhaObeSam17}) and the cited references therein.

Recall that for a fixed $n\in \mathbb{N}$, the regularized output least-squares functional is given by
\begin{equation*}
    \widehat{J}_{\kappa_n}(a):=\frac12\|u_{\varsigma_n}(a)-z_{\delta_n}\|^2_Z+\kappa_n R(u_{\varsigma_n}),
\end{equation*}
where $R$ is a smooth regularizer and $u_{\varsigma_n}(a)$ is the unique solution of the regularized problem \eqref{RVP}, that is,
\begin{equation}\label{RVP2}T_{\tau_n}(a,u_{\varsigma_n}(a),v)+\epsilon_n S(u_{\varsigma_n}(a),v)=m_{\nu_n}(v)+\epsilon_n\ell_{\delta_n}(v),\ \ \text{for every}\ v\in V.\end{equation}
Here $\{\epsilon_n\}$, $\{\tau_n\}$, $\{\kappa_n\}$, $\{\delta_n\}$, and $\{\nu_n\}$ are sequence of  positive reals, $\ell\in V^*$, and for each $n\in \mathbb{N}$, $m_{\nu_n}\in V^*$, $\ell_{\delta_n}\in V^*$, and  $z_{\delta_n}\in  Z$ satisfies \eqref{Noise}. Moreover, $T_{\tau_n}:B\times V\times V\to \mathbb{R}$ satisfies \eqref{TNoise} and the bilinear and symmetric form  $S:V\times V\to \mathbb{R}$ satisfies \eqref{SCond}.

By using the chain rule, the derivative of $\widehat{J}_{\kappa_n}$ at $a \in A$ in any direction $\delta a $ is given by
\begin{equation*}
    D\widehat{J}_{\kappa_n}(a)(\delta a)=\left\langle Du_{\varsigma_n}(a)(\delta a),u_{\varsigma_n}(a)-z_{\delta_n}\right\rangle_Z+\kappa_n DR(a)(\delta a),
\end{equation*}
where $Du_{\varsigma_n}(a)(\delta a)$ is the derivative of the regularized parameter-to-solution map $u_{\varsigma_n}$ and  $DR(a)(\delta a)$ is the derivative of the regularizer $R$, both computed at $a $ in the direction $\delta a $.

For an arbitrary $v\in V$, we define the functional $L_{\kappa_n}:B\times V\to \mathbb{R}$ by
\begin{equation*}
    L_{\kappa_n}(a,v)=\widehat{J}_{\kappa_n}(a)+T_{\tau_n}(a,u_{\varsigma_n}(a),v)+\epsilon_n S(u_{\varsigma_n}(a),v)-m_{\nu_n}(v)-\epsilon_n\ell_{\delta_n}(v).
\end{equation*}
Since $u_{\varsigma_n}(a)$ solves  \eqref{RVP2}, for every $v\in V$, we have $L_{\kappa_n}(a,v)=J_{\kappa_n}(a)$, and consequently, for every $v\in V$ and for every  direction $\delta a$, we have
\begin{equation*}
    \partial_{a }L_{\kappa_n}(a,v)\left( \delta a  \right) =D\widehat{J}_{\kappa_n}(a)\left( \delta a  \right).
\end{equation*}
The key idea for the first-order adjoint method is to choose $v$ to bypass a  direct computation of $Du_{\varsigma_n}(a )(\delta  a )$.n
To understand a choice of $v$, we compute
\begin{multline}
    \partial_{a } L_{\kappa_n}(a,v)\left( \delta
    a  \right)=\left\langle Du_{\varsigma_n}(a)(\delta a ),u_{\varsigma_n}-z_{\delta_n}\right\rangle_Z +\kappa_n DR(a)(\delta a)\\
    +T_{\tau_n}(\delta a,u_{\varsigma_n},v)+T_{\tau_n}(a,Du_{\varsigma_n}(a )(\delta a ),v)+\epsilon_nS( Du_{\varsigma_n}(a )(\delta a ),v). \label{CR2}
\end{multline}
For $a \in A$, and $u_{\varsigma_n}$ satisfying \eqref{RVP}, let $w_{\varsigma_n}(a)$ be the unique solution of the adjoint problem
\begin{equation}\label{ASPP}
    T_{\tau_n}(a,w_{\varsigma_n},v)+\epsilon_nS( w_{\varsigma_n},v) =\left\langle z_{\delta_n}-u_{\varsigma_n},v\right\rangle_Z,\quad \text{for every}\ v\in V.
\end{equation}
We set $v=Du_{\varsigma_n}(a )(\delta a )$ in the above equation and use the symmetry of $T_{\tau_n}$ and $S$ to obtain
\begin{equation}\label{ASPP1}
    T_{\tau_n}(a,Du_{\varsigma_n}(a )(\delta a ),w_{\varsigma_n})+\epsilon_n S( Du_{\varsigma_n}(a )(\delta a ),w_{\varsigma_n}) +\left\langle u_{\varsigma_n}-z_{\delta_n},Du_{\varsigma_n}(a )(\delta a )\right\rangle_Z=0.
\end{equation}
By plugging $v=w_{\varsigma_n}$ in \eqref{CR2}, using \eqref{ASPP1}, we obtain
\begin{equation*}
    \begin{aligned}
        \partial_{a }L_{\kappa_n}(a,w_{\varsigma_n})\left( \delta
        a  \right)&=\left\langle Du_{\varsigma_n}(a)(\delta a ),u_{\varsigma_n}-z_{\delta_n}\right\rangle_Z +\kappa_n DR(a)(\delta a)\\
        \MoveEqLeft[-1]+T_{\tau_n}(\delta a,u_{\varsigma_n},w_{\varsigma_n})+T_{\tau_n}(a,Du_{\varsigma_n}(a )(\delta a ),w_{\varsigma_n})+\epsilon_n S(Du_{\varsigma_n}(a )(\delta a ),w_{\varsigma_n}) \\
        &=\kappa_n DR(a)(\delta a)+T_{\tau_n}(\delta a,u_{\varsigma_n},w_{\varsigma_n}),
    \end{aligned}
\end{equation*}
which gives the following formula for the first-order derivative of $J_{\kappa_n}$:
\begin{equation}\label{FD-OLS}
    D\widehat{J}_{\kappa_n}(a)\left( \delta a  \right) =\kappa_n DR(a)(\delta a)+T_{\tau_n}(\delta a,u_{\varsigma_n},w_{\varsigma_n}).
\end{equation}
In summary, the following scheme computes $D\widehat{J}_{\kappa_n}(a)\left( \delta a  \right) $ for the given direction $\delta a $:
\begin{enumerate}[label=\arabic*., noitemsep]
    \item Compute $u_{\varsigma_n}(a) $ by using \eqref{RVP2}.
    \item Compute $w_{\varsigma_n}(a)$ by using \eqref{ASPP}.
    \item Compute $D\widehat{J}_{\kappa_n}(a)\left( \delta a  \right) $ by using \eqref{FD-OLS}.
\end{enumerate}

We will now derive a second-order adjoint  method for the evaluation of the second-order derivative of the regularized OLS functional. The goal is to derive a formula for the second-order  derivative that does not require the second-order derivative of the regularized parameter-to-solution map. The central idea is to compute  $\delta u_{\varsigma_n}$ directly by using  \eqref{var4RP} and bypass the computation of $\delta^{2}u_{\varsigma_n}$ by an adjoint approach.

Given a fixed direction $\delta a _{2}$ and an arbitrary $v\in V$, we define
\begin{equation*}
    \begin{aligned}
        L_{\kappa_n}(a,v)&=D\widehat{J}_{\kappa_n}(a)(\delta a  _{2})+T_{\tau_n}(a,Du_{\varsigma_n}(a)\delta a_2,v)+\epsilon_n S( Du_{\varsigma_n}(a)\delta a_2,v)+T_{\tau_n}(\delta a_2,u_{\varsigma_n},v)\\
                         &=\left\langle Du_{\varsigma_n}(a)(\delta a_{2}),u_{\varsigma_n}-z_{\delta_n}\right\rangle_Z+\kappa_n DR(a)(\delta a_{2})+T_{\tau_n}(a,Du_{\varsigma_n}(a)\delta a_2,v)\\
        \MoveEqLeft[-1]+\epsilon_n S( Du_{\varsigma_n}(a)\delta a_2,v) +T_{\tau_n}(\delta a_2,u_{\varsigma_n},v).
    \end{aligned}
\end{equation*}
Since $L_{\kappa_n}(a,v)=D\widehat{J}_{\kappa_n}(a)(\delta a_{2})$, for any $v\in  V$,  and hence for any $\delta a  _{1}$, we have 
\begin{equation*}
    \partial_{a } L_{\kappa_n}(a,v)(\delta a
    _{1})=D^{2}\widehat{J}_{\kappa_n}(a)(\delta a  _{1},\delta a  _{2}).
\end{equation*}
We compute the derivative  of $L_{\kappa_n}$ in the direction $\delta a _1$ as follows
\begin{equation*}
    \begin{aligned}
        \partial_{a } L_{\kappa_n}(a,v)(\delta a_{1})&=\left\langle D^{2}u_{\varsigma_n}(a)(\delta a_{1},\delta a  _{2}),u_{\varsigma_n}-z\right\rangle_Z
        +\left\langle Du_{\varsigma_n}(a)(\delta a_{2}),Du_{\varsigma_n}(a)(\delta a_{1})\right\rangle_Z\\
        \MoveEqLeft[-1]+\kappa_n D^{2}R(a)(\delta a_{1},\delta a_{2})+T_{\tau_n}(\delta a_1,Du_{\varsigma_n}(a)\delta a_2,v)\\
        \MoveEqLeft[-9.75]+T_{\tau_n}(a,D^2u_{\varsigma_n}(a)(\delta a_1,\delta a_2),v)\\
        \MoveEqLeft[-1]+\epsilon_n \langle D^2u_{\varsigma_n}(a)(\delta a_1,\delta a_2),v\rangle_Z+T_{\tau_n}(\delta a_2,Du_{\varsigma_n}(a)\delta a_1,v).
    \end{aligned}
\end{equation*}
Let $w_{\varsigma_n}(a)$ be the solution of the  adjoint problem \eqref{ASPP}.
We set $v=D^2u_{\varsigma_n}(a )(\delta a_1,\delta a_2 )$ in \eqref{ASPP} and use the symmetry of $T_{\tau_n}$ and $S$ to obtain
\begin{multline}\label{ASPP2}
    T_{\tau_n}(a,D^2u_{\varsigma_n}(a )(\delta a_1,\delta a_2 ),w_{\varsigma_n})+\epsilon_n S( D^2u_{\varsigma_n}(a )(\delta a_1,\delta a_2 ),w_{\varsigma_n})  \\
    +\left\langle u_{\varsigma}-z_{\delta_n},D^2u_{\varsigma_n}(a )(\delta a_1,\delta a_2 )\right\rangle_Z=0.
\end{multline}
Using \eqref{ASPP2}, we have
\begin{equation*}
    \begin{aligned}
        \partial_{a }L_{\kappa_n}(a,w_{\varsigma_n})(\delta a_{1}) &=\left\langle D^{2}u_{\varsigma_n}(a)(\delta a_{1},\delta a_{2}),u_{\varsigma_n}-z\right\rangle_Z
        +\left\langle Du_{\varsigma_n}(a)(\delta a_{2}),Du_{\varsigma_n}(a)(\delta a_{1})\right\rangle_Z\\
        \MoveEqLeft[-1]+\kappa_n D^{2}R(a)(\delta a_{1},\delta a_{2})+T_{\tau_n}(\delta a_1,Du_{\varsigma_n}(a)\delta a_2,w_{\varsigma_n})\\
        \MoveEqLeft[-1]+T_{\tau_n}(a,D^2u_{\varsigma_n}(a)(\delta a_1,\delta a_2),w_{\varsigma_n})\\
        \MoveEqLeft[-1]+\epsilon_n S( D^2u_{\varsigma_n}(a)(\delta a_1,\delta a_2),w_{\varsigma_n})+T_{\tau_n}(\delta a_2,Du_{\varsigma_n}(a)\delta a_1,w_{\varsigma_n})\\
        &=\kappa_n D^{2}R(a)(\delta a  _{1},\delta a  _{2})+\left\langle Du_{\varsigma_n}(a)(\delta a  _{2}),Du_{\varsigma_n}(a)(\delta a
    _{1})\right\rangle_Z \\
    \MoveEqLeft[-1] +T_{\tau_n}(\delta a_1,Du_{\varsigma_n}(a)\delta a_2,w_{\varsigma_n})
    +T_{\tau_n}(\delta a_2,Du_{\varsigma_n}(a)\delta a_1,w_{\varsigma_n}),
\end{aligned}
\end{equation*}
and consequently, we have
\begin{multline*}
    D^{2}\widehat{J}_{\kappa_n}(a)(\delta a_{1},\delta a_{2})=\kappa_n D^{2}R(a)(\delta a  _{1},\delta a  _{2})+\left\langle Du_{\varsigma_n}(a)(\delta a  _{2}),Du_{\varsigma_n}(a)(\delta a
    _{1})\right\rangle_Z \\
    +T_{\tau_n}(\delta a_1,Du_{\varsigma_n}(a)\delta a_2,w_{\varsigma_n})+T_{\tau_n}(\delta a_2,Du_{\varsigma_n}(a)\delta a_1,w_{\varsigma_n}).
\end{multline*}
In particular,
\begin{multline}
    D^{2}\widehat{J}_{\kappa_n}(a)(\delta a,\delta a)=\kappa_n D^{2}R(a)(\delta a,\delta a)+\left\langle Du_{\varsigma_n}(a)(\delta a),Du_{\varsigma_n}(a)\delta a\right\rangle_Z \\
    +2T_{\tau_n}(\delta a,Du_{\varsigma_n}(a)\delta a,w_{\varsigma_n}).\label{SD1-OLS}
\end{multline}
In summary, the following scheme computes $D^{2}\widehat{J}_{\kappa_n}(a)(\delta a,\delta a)$ for any direction $\delta a$:
\begin{enumerate}[label=\arabic*., noitemsep]
    \item Compute $u_{\varsigma_n}(a)$ by \eqref{RVP2}.
    \item Compute $Du_{\varsigma_n}(a)(\delta a)$ by \eqref{var4RP}.
    \item Compute $w_{\varsigma_n}(a)$ by \eqref{ASPP}.
    \item Compute $D^{2}\widehat{J}_{\kappa_n}(a)(\delta a,\delta a)$ by \eqref{SD1-OLS}.
\end{enumerate}
The above schemes yield efficient formulas for gradient and Hessian computation.

\section{Computational Framework}\label{sec:computational}

In this section, we develop a finite element method based discretization framework for the direct and the inverse problems.
Let $\Th$ be a triangulation of the domain $\Omega $.  We define
$\Ah$ to be the space of all continuous piecewise polynomials of degree
$d_a$ relative to $\Th$.  Similarly, $\Uh$ will be the space of all
continuous piecewise polynomials of degree $d_u$ relative to $\Th$.
Bases for $\Ah$ and $\Uh$ will be represented by $\left\{\psi_1,\psi_2,\ldots,\psi_m\right\}$ and
$\left\{\psi_1,\varphi_2,\ldots,\varphi_n\right\}$, respectively.  The space $\Ah$ is then isomorphic to $\mathbb{R}^m$, and for any
$a\in\Ah$, we define $A\in\mathbb{R}^m $ by $A_i:=a(x_i),\ i=1,2,\ldots,m$, where $\{\psi_1,\psi_2,\ldots,\psi_m\}$ is a nodal basis corresponding to the nodes $\{x_1,x_2,\ldots,x_m\}$.  Conversely, each
$A\in\mathbb{R}^m $ corresponds to $a\in\Ah$ defined by $a:=\sum_{i=1}^mA_i\psi_i$. Similarly, $u\in\Uh$ will correspond to $U\in\mathbb{R}^{n}$, where $U_i:=u(y_i),\ i=1,2,\ldots,n$ and $u=\sum_{i=1}^nU_i\varphi_i$. Here $y_1,y_2,\ldots,y_n$ are the nodes of the mesh defining $\Uh$. Note that although both $\Ah$ and $\Uh$ are defined relative to the same triangles, the nodes are different.

For a fixed $\varsigma_n$, we define $\mathbb{F}_{\varsigma_n}:\mathbb{R}^m \rightarrow\mathbb{R}^{n}$ to be the finite element solution operator
assigning a coefficient $a\in\Ah$ to the approximate solution $v\in\Uh$.
Then $\mathbb{F}_{\varsigma_n}(A)=V_{\varsigma_n}$, where $V_{\varsigma_n}$ is defined by
\begin{equation}\label{vdef}
    \left[K_{\tau_n}(A)+\epsilon_nW\right]V_{\varsigma_n}=P_{\delta_n}
\end{equation}
and $K_{\tau_n}(A)\in\mathbb{R}^{n\times n}$ is the stiffness matrix, $W$ is the matrix generated by $S$, and $P_{\delta_n}\in\mathbb{R}^{n}$ is the load
vector:
\begin{align*}
    K_{\tau_n}(A)_{ij}&=T_{\tau_n}(a,\varphi_i,\varphi_j), &i,j=1,2,\ldots,n,\\
    W_{ij}&=S(\varphi_i,\varphi_j), &i,j=1,2,\ldots,n,\\
    P^i_{\delta_n}&=m_{\nu_n}(\varphi_i)+\epsilon_n\ell_{\delta_n}(\varphi_i), &i=1,2,\ldots,n.
\end{align*}
For future reference, it will be useful to notice that $K_{\tau_n}(A)_{ij}=T_{ijk}A_k$,
where the summation convention is used and $T$ is the tensor defined by
\begin{equation*}
    T_{ijk}=T_{\tau_n}(\psi_k,\varphi_i,\varphi_j),\quad i,j=1,2,\ldots,n,\quad
    k=1,2,\ldots,m.
\end{equation*}
The derivative of the regularized parameter-to-solution map is easily computed as
\begin{equation*}
    D\mathbb{F}_{\varsigma_n}(A)(\delta A):=\dV_{\varsigma_n}=-\left[K_{\tau_n}(A)+\epsilon_nW\right]^{-1}K_{\tau_n}(\dA)V_{\varsigma_n}.
\end{equation*}
To write the formula for $\dV_{\varsigma_n}$ in a tractable form, we define
the matrix $L_{\tau_n}(\mathcal{V})$ by the condition
\begin{equation*}
    L_{\tau_n}(\mathcal{V})A=K_{\tau_n}(A)\mathcal{V}\text{ for all } A\in\mathbb{R}^m ,\mathcal{V}\in\mathbb{R}^{n}.
\end{equation*}
Using this notation,
\begin{equation*}
    \dV_{\varsigma_n}=-\left[K_{\tau_n}(A)+\epsilon_nW\right]^{-1}L_{\tau_n}(V_{\varsigma_n})\dA.
\end{equation*}

\subsection{Discrete OLS}\label{sec:computational:ols}

In the following, for simplicity, we ignore a discretization of the the regularization term. Using the same notation as above, we can define the $L^2$-output least-squares
objective function by
\begin{equation}\label{falkJ}
    \widehat{J}_{n}(A)=\frac{1}{2}(V_{\varsigma_n}-Z_{\delta_n})^T M(V_{\varsigma_n}-Z_{\delta_n}),
\end{equation}
where $V_{\varsigma_n}$ solves \eqref{vdef} and $M$ is the mass matrix.

The gradient of $\widehat{J}_{\kappa_n}$ can be computed as follows
\begin{equation*}
    \begin{aligned}
        D\widehat{J}_{n}(A)\dA&=\dV_{\varsigma_n}^T M(V_{\varsigma_n}-Z_{\delta_n})\\
                              &=-\left(\left[K_{\tau_n}(A)+\epsilon_nW\right]^{-1}L_{\tau_n}(V_{\varsigma_n})\dA\right)^T M(V_{\varsigma_n}-Z_{\delta_n})\\
                              &=-\dA^T L_{\tau_n}(V_{\varsigma_n})^T\left[K_{\tau_n}(A)+\epsilon_nW\right]^{-1}M(V_{\varsigma_n}-Z_{\delta_n}),
    \end{aligned}
\end{equation*}
and hence
\begin{equation*}
    \grad \widehat{J}_{n}(A)=-L_{\tau_n}(V_{\varsigma_n})^T\left[K_{\tau_n}(A)+\epsilon_nW\right]^{-1}M(V_{\varsigma_n}-Z_{\delta_n}).
\end{equation*}
We shall now proceed to compute the Hessian.  First, we observe that
\begin{equation*}
    \begin{aligned}
        D^2\widehat{J}(A)(\dA,\dA)&= \dA^T(-L_{\tau_n}(\dV_{\varsigma_n})^T)\left[K_{\tau_n}(A)+\epsilon_nW\right]^{-1}M(V_{\varsigma_n}-Z_{\delta_n})\\
        \MoveEqLeft[-1]+\delta A^TL_{\tau_n}(V_{\varsigma_n})^T\left[K_{\tau_n}(A)+\epsilon_nW\right]^{-1}K_{\tau_n}(\delta A) \left[K_{\tau_n}(A)+\epsilon_nW\right]^{-1}M(V_{\varsigma_n}-Z_{\delta_n})\\
        \MoveEqLeft[-1]-\delta A^TL_{\tau_n}(V_{\varsigma_n})^T\left[K_{\tau_n}(A)
        +\epsilon_nW\right]^{-1}M\dV_{\varsigma_n}),
    \end{aligned}
\end{equation*}
where we use the formula
\begin{equation*}
    \begin{aligned}
        D\left[K_{\tau_n}(A)+\epsilon_nW\right]^{-1}\dA&=-\left[K_{\tau_n}(A)+\epsilon_nW\right]^{-1}DK_{\tau_n}(A)\dA \left[K_{\tau_n}(A)+\epsilon_nW\right]^{-1}\\
                                                       &=-\left[K_{\tau_n}(A)+\epsilon_nW\right]^{-1}K_{\tau_n}(\delta A) \left[K_{\tau_n}(A)+\epsilon_nW\right]^{-1}.
    \end{aligned}
\end{equation*}
We will simplify all the the terms involved in the formula for the second-order derivative. First of all, we have
\begin{equation*}
    \begin{aligned}
        L_{\tau_n}(V_{\varsigma_n})^T&\left[K_{\tau_n}(A)+\epsilon_nW\right]^{-1}DK_{\tau_n}(A)(\delta A) \left[K_{\tau_n}(A)+\epsilon_nW\right]^{-1}M(V_{\varsigma_n}-Z_{\delta_n})\\
                                     &=L_{\tau_n}(V_{\varsigma_n})^T\left[K_{\tau_n}(A)+\epsilon_nW\right]^{-1}K_{\tau_n}(\delta A)\left[K_{\tau_n}(A)+\epsilon_nW\right]^{-1}M(V_{\varsigma_n}-Z_{\delta_n})\\
                                     &=L_{\tau_n}(V_{\varsigma_n})^T\left[K_{\tau_n}(A)+\epsilon_nW\right]^{-1}
        L_{\tau_n}(\left[K_{\tau_n}(A)+\epsilon_nW\right]^{-1}M(V_{\varsigma_n}-Z_{\delta_n}))\delta A,
    \end{aligned}
\end{equation*}
which, by using the fact that
\begin{equation*}
    L_{\tau_n}(\mathcal{V})^TU=L_{\tau_n}(U)^T\mathcal{V},\quad \text{for all}\ U,\mathcal{V}\in\mathbb{R}^{n}
\end{equation*}
gives
\begin{equation*}
    \begin{aligned}
        -L_{\tau_n}(\dV_{\varsigma_n})^T&\left[K_{\tau_n}(A)+\epsilon_nW\right]^{-1}M(V_{\varsigma_n}-Z_{\delta_n})\\
                                        &=L_{\tau_n}(\left[K_{\tau_n}(A)+\epsilon_nW\right]^{-1}L_{\tau_n}(V_{\varsigma_n})\dA)^T\left[K_{\tau_n}(A)
    +\epsilon_nW\right]^{-1}M(V_{\varsigma_n}-Z_{\delta_n})\\
    &=L_{\tau_n}(\left[K_{\tau_n}(A)
+\epsilon_nW\right]^{-1}M(V_{\varsigma_n}-Z_{\delta_n}))^T\left[K_{\tau_n}(A)+\epsilon_nW\right]^{-1}L_{\tau_n}(V_{\varsigma_n})\dA.
    \end{aligned}
\end{equation*}
Finally,
\begin{multline*}
    -L_{\tau_n}(V_{\varsigma_n})^T\left[K_{\tau_n}(A)
    +\epsilon_nW\right]^{-1}M\dV_{\varsigma_n}\\
    = L_{\tau_n}(V_{\varsigma_n})^T\left[K_{\tau_n}(A)
    +\epsilon_nW\right]^{-1}M\left[K_{\tau_n}(A)+\epsilon_nW\right]^{-1}L_{\tau_n}(V_{\varsigma_n})\dA.
\end{multline*}
Combining the above results, we obtain that
\begin{equation*}
    \begin{aligned}
        \hess \widehat{J}_{\kappa_n}(A)&=L_{\tau_n}(V_{\varsigma_n})^T\left[K_{\tau_n}(A)+\epsilon_nW\right]^{-1}
        L_{\tau_n}(\left[K_{\tau_n}(A)+\epsilon_nW\right]^{-1}M(V_{\varsigma_n}-Z_{\delta_n}))\\
        \MoveEqLeft[-1]+L_{\tau_n}(\left[K_{\tau_n}(A)
        +\epsilon_nW\right]^{-1}M(V_{\varsigma_n}-Z_{\delta_n}))^T\left[K_{\tau_n}(A)+\epsilon_nW\right]^{-1}L_{\tau_n}(V_{\varsigma_n})\\
        \MoveEqLeft[-1]+L_{\tau_n}(V_{\varsigma_n})^T\left[K_{\tau_n}(A)
        +\epsilon_nW\right]^{-1}M\left[K_{\tau_n}(A)+\epsilon_nW\right]^{-1}L_{\tau_n}(V_{\varsigma_n}).
    \end{aligned}
\end{equation*}

We emphasize that, after a discretization, the first-order and  second-order adjoint formulae discussed in the previous section would lead to an alternative scheme for computing the gradient and the Hessian of the OLS objective.

\subsection{Discrete MOLS}\label{sec:computational:mols}

The discrete MOLS function $J_{n}:\mathbb{R}^m \rightarrow\mathbb{R}$ is given by
\begin{equation*}
    J_{n}(A)=\frac{1}{2}(V_{\varsigma_n}-Z_{\delta_n})^T K_{\tau_n}(A)(V_{\varsigma_n}-Z_{\delta_n})+\frac{\epsilon_n}{2}\left(V_{\varsigma_n}-Z_{\delta_n}\right)^T W\left(V_{\varsigma_n}-Z_{\delta_n}\right),
\end{equation*}
where  $V_{\varsigma_n}$ solves \eqref{vdef}, $W$ is the symmetric matrix generated by the bilinear form $S$, and $Z_{\delta_n}$ is the discrete data.

We can now compute the gradient as follows
\begin{equation*}
    \begin{aligned}
        DJ_{n}(A)\dA&:=\dV_{\varsigma_n}^T K_{\tau_n}(A)(V_{\varsigma_n}-Z_{\delta_n})+\frac{1}{2}(V_{\varsigma_n}-Z_{\delta_n})^T DK_{\tau_n}(A)\dA(V_{\varsigma_n}-Z_{\delta_n})\\
        \MoveEqLeft[-1]+\epsilon_n\dV_{\varsigma_n}^T W(V_{\varsigma_n}-Z_{\delta_n})\\
        &=\dV_{\varsigma_n}^T \left[K_{\tau_n}(A)+\epsilon_nW\right](V_{\varsigma_n}-Z_{\delta_n})+\frac{1}{2}(V_{\varsigma_n}-Z_{\delta_n})^T K_{\tau_n}(\dA)(V_{\varsigma_n}-Z_{\delta_n})\\
        &=\left[-\left[K_{\tau_n}(A)+\epsilon_nW\right]^{-1}L_{\tau_n}(V_{\varsigma_n})\dA\right]^T \left[K_{\tau_n}(A)+\epsilon_nW\right](V_{\varsigma_n}-Z_{\delta_n})\\
        \MoveEqLeft[-1]+\frac{1}{2}(V_{\varsigma_n}-Z_{\delta_n})^T K(\dA)(V_{\varsigma_n}-Z_{\delta_n})\\
        &=-\dA^T L_{\tau_n}(V_{\varsigma_n})^T(V_{\varsigma_n}-Z_{\delta_n})+\frac{1}{2}(V_{\varsigma_n}-Z_{\delta_n})^T L_{\tau_n}(V_{\varsigma_n}-Z_{\delta_n})\dA\\
        &=-\dA^T L_{\tau_n}(V_{\varsigma_n})^T(V_{\varsigma_n}-Z_{\delta_n})+\frac{1}{2}\dA^T L_{\tau_n}(V_{\varsigma_n}-Z_{\delta_n})^T(V_{\varsigma_n}-Z_{\delta_n})\\
        &=-\frac{1}{2}\dA^T L_{\tau_n}(V_{\varsigma_n}+Z_{\delta_n})^T(V_{\varsigma_n}-Z_{\delta_n}),
    \end{aligned}
\end{equation*}
which yields
\begin{equation*}
    \grad J_{n}(A)=-\frac{1}{2}L_{\tau_n}(V_{\varsigma_n}+Z_{\delta_n})^T(V_{\varsigma_n}-Z_{\delta_n})
    =-\frac{1}{2}L_{\tau_n}(V_{\varsigma_n})^TV_{\varsigma_n}
    +\frac{1}{2}L_{\tau_n}(Z_{\delta_n})^TZ_{\delta_n}.
\end{equation*}
For the second-order derivative, from 
\begin{equation*}
    DJ_{n}(A)\dA=-\frac{1}{2}\dA^T L_{\tau_n}(V_{\varsigma_n})^TV_{\varsigma_n}+\frac{1}{2}\dA^T L_{\tau_n}(Z_{\delta_n})^TZ_{\delta_n}
\end{equation*}
we evaluate
\begin{equation*}
    D^2J_{n}(A)(\dA,\dA)=-\frac{1}{2}\dA^T L_{\tau_n}(\dV_{\varsigma_n})^TV_{\varsigma_n}-\frac{1}{2}\dA^T L_{\tau_n}(V_{\varsigma_n})^T\dV_{\varsigma_n}.
\end{equation*}
Since $L_{\tau_n}(U)^T\mathcal{V}=L_{\tau_n}(\mathcal{V})^TU$ for all $U,\mathcal{V}\in \mathbb{R}^n$,
and hence,
\begin{equation*}
    L_{\tau_n}(\dV_{\varsigma_n})^TV_{\varsigma_n} = L_{\tau_n}(V_{\varsigma_n})^T\dV_{\varsigma_n} = -L_{\tau_n}(V_{\varsigma_n})^T\left[K_{\tau_n}(A)+\epsilon_nW\right]^{-1}L_{\tau_n}(V)\dA.
\end{equation*}
Consequently,
\begin{equation*}
    \begin{aligned}
        D^2J_{n}(A)(\dA,\dA)&=-\dA^T L_{\tau_n}(V_{\varsigma_n})^T\dV_{\varsigma_n}\\
                            &=\dA^T L_{\tau_n}(V_{\varsigma_n})^T\left[K_{\tau_n}(A)+\epsilon_nW\right]^{-1}L_{\tau_n}(V_{\varsigma_n})\dA,
    \end{aligned}
\end{equation*}
which shows that
\begin{equation*}
    \hess J_{n}(A)=L_{\tau_n}(V_{\varsigma_n})^T\left[K_{\tau_n}(A)+\epsilon_nW\right]^{-1}L_{\tau_n}(V_{\varsigma_n}).
\end{equation*}

Summarizing, the necessary formulas are 
\begin{align*}
    J_{n}(A)&=\frac{1}{2}(V_{\varsigma_n}-Z_{\delta_n})^T \left[K_{\tau_n}(A)+\epsilon_nW\right](V_{\varsigma_n}-Z_{\delta_n}),\\
    \grad J_{n}(A)&=-\frac{1}{2}L_{\tau_n}(V_{\varsigma_n})^TV_{\varsigma_n}
    +\frac{1}{2}L_{\tau_n}(Z_{\delta_n})^TZ_{\delta_n},\\
    \hess J_{n}(A)&=L_{\tau_n}(V_{\varsigma_n})^T\left[K_{\tau_n}(A)+\epsilon_nW\right]^{-1}L_{\tau_n}(V_{\varsigma_n}).
\end{align*}

\section{Computational Experiments}\label{sec:experiments}

We now report preliminary numerical experiments to demonstrate the feasibility of the proposed framework. We will identify the coefficient $a$ in the Neumann boundary value problem \eqref{Neumann}. Our experiments are of synthetic nature, and hence the data vectors are computed, not measured. We solve numerically the regularized problems \eqref{RMinOLS} and  \eqref{MOLSNC} by using the piecewise linear finite elements. We used the finite element library FreeFem++ \cite{hecht2012new}.  For simplicity, we used the $H^{1}(\Omega )$ norm as the regularizer. The choose the regularization parameters $\kappa_{n}$ and $\varepsilon _{n}$ by trial and error. The elliptic regularization of the variational problem was crucial in the identification process. As expected, the reconstruction process failed for $\varepsilon _{n}=0$ (see \cref{Fig:FreeNoisekappa0001eps0}). Finding a stable numerical solution of a pure Neumann problem is quite challenging (see \cite{BocLeh05,Dai07})  and our computations show that the elliptic regularization does a remarkable job in giving a stable solution.

The (normalized) unique solution is $\bar{u}(x_{1},x_{2})=\cos
(\pi x_{1}^{2})\cos (2\pi x_{2})$ and $\bar{a}(x_{1},x_{2})=1$.  In our experiment, both the OLS objective and the MOLS objective gave quite a satisfactory reconstruction,  see \cref{Table_OLS_Neumann,Table_MOLS_Neumann} and
\cref{Fig:OLSFreeNoisekappa0001eps0001,Fig:MOLSFreeNoisekappa0001eps0001}.  To study the influence of noise, we considered contaminated  data $z_{\delta _{n}}=z+\delta _{n}\eta (t),$ with $\eta (t)$ uniformly distributed in $[0,1]$. The reconstruction is again quite stable, as seen in \cref{Table_OLS_NeumannNoisy} and
\cref{Fig:kappa0001eps0001delta01}.
\begin{table}[tp]
    \caption{Reconstruction error for the OLS  for different discretization level and $\kappa=\epsilon=0.0001$.}
    \label{Table_OLS_Neumann}\centering
    \begin{tabular}{
            S[table-format=1.7,group-digits=false]
            S[table-format=1.2e-1,scientific-notation=true]
            S[table-format=1.2e-1,scientific-notation=true]
            S[table-format=1.2e-1,scientific-notation=true]
            S[table-format=1.2e-1,scientific-notation=true]
        }
        \toprule
        {$h$} & 
        {$\frac{\left\Vert a^{h}-I_{h}\bar{a}_{0}\right\Vert _{L^{2}(\Omega )}}{\left\Vert I_{h}\bar{a}_{0}\right\Vert _{L^{2}(\Omega )}}$} &
        {$\frac{\left\Vert u^{h}-u_{0}^{h}\right\Vert _{L^{2}(\Omega )}}{\left\Vert u_{0}^{h}\right\Vert _{L^{2}(\Omega )}}$} & {$\frac{\left\Vert a^{h}-I_{h}\bar{a}_{0}\right\Vert _{L^{\infty }(\Omega )}}{\left\Vert I_{h}\bar{a}_{0}\right\Vert _{L^{\infty }(\Omega )}}$} & 
        {$\frac{\left\Vert u^{h}-u_{0}^{h}\right\Vert _{L^{\infty }(\Omega )}}{\left\Vert u_{0}^{h}\right\Vert _{L^{\infty }(\Omega )}}$} \\ 
        \midrule
        0.0471405 & 1.13e-02 & 2.13e-03 & 3.34e-02 & 9.61e-03 \\ 
        0.0353553 & 6.96e-03 & 1.27e-03 & 1.91e-02 & 6.66e-03 \\ 
        0.0282843 & 5.05e-03 & 8.90e-04 & 1.38e-02 & 4.87e-03 \\ 
        0.0235702 & 4.03e-03 & 7.19e-04 & 9.76e-03 & 3.83e-03 \\ 
        0.0202031 & 3.34e-03 & 6.11e-04 & 8.24e-03 & 2.98e-03 \\ 
        0.0176777 & 3.23e-03 & 6.07e-04 & 8.65e-03 & 2.43e-03 \\ 
        \bottomrule
    \end{tabular}
\end{table}
\begin{table}[tp]
    \caption{Reconstruction error for the MOLS for different discretization level and  $\kappa=0.01, \epsilon=0.0001.$ }
    \label{Table_MOLS_Neumann}\centering
    \begin{tabular}{
            S[table-format=1.7,group-digits=false]
            S[table-format=1.2e-1,scientific-notation=true]
            S[table-format=1.2e-1,scientific-notation=true]
            S[table-format=1.2e-1,scientific-notation=true]
            S[table-format=1.2e-1,scientific-notation=true]
        }
        \toprule
        {$h$} & 
        {$\frac{\left\Vert a^{h}-I_{h}\bar{a}_{0}\right\Vert _{L^{2}(\Omega )}}{\left\Vert I_{h}\bar{a}_{0}\right\Vert _{L^{2}(\Omega )}}$} &
        {$\frac{\left\Vert u^{h}-u_{0}^{h}\right\Vert _{L^{2}(\Omega )}}{\left\Vert u_{0}^{h}\right\Vert _{L^{2}(\Omega )}}$} & {$\frac{\left\Vert a^{h}-I_{h}\bar{a}_{0}\right\Vert _{L^{\infty }(\Omega )}}{\left\Vert I_{h}\bar{a}_{0}\right\Vert _{L^{\infty }(\Omega )}}$} & 
        {$\frac{\left\Vert u^{h}-u_{0}^{h}\right\Vert _{L^{\infty }(\Omega )}}{\left\Vert u_{0}^{h}\right\Vert _{L^{\infty }(\Omega )}}$} \\ 
        \midrule
        0.0471405 & 9.54e-03 & 4.37e-03 & 4.32e-02 & 1.16e-02 \\ 
        0.0353553 & 5.83e-03 & 2.50e-03 & 2.50e-02 & 7.50e-03 \\ 
        0.0282843 & 4.24e-03 & 1.66e-03 & 1.70e-02 & 5.49e-03 \\ 
        0.0235702 & 3.34e-03 & 1.22e-03 & 1.23e-02 & 4.16e-03 \\ 
        0.0202031 & 2.82e-03 & 1.04e-03 & 9.77e-03 & 3.54e-03 \\ 
        0.0176777 & 2.36e-03 & 9.21e-04 & 8.18e-03 & 3.24e-03 \\ 
        \bottomrule
    \end{tabular}
\end{table}
\begin{table}[tp]
    \caption{Reconstruction error for the OLS for different noise levels $\delta_n$ for $h=0.0176777$, $\kappa=\epsilon=0.0001$.}
    \label{Table_OLS_NeumannNoisy}\centering
    \begin{tabular}{
            S[table-format=1e-1,scientific-notation=true]
            S[table-format=1.2e-1,scientific-notation=true]
            S[table-format=1.2e-1,scientific-notation=true]
            S[table-format=1.2e-1,scientific-notation=true]
            S[table-format=1.2e-1,scientific-notation=true]
        }
        \toprule
        {$\delta_n$} & 
        {$\frac{\left\Vert a^{h}-I_{h}\bar{a}\right\Vert _{L^{2}(\Omega )}}{\left\Vert I_{h}\bar{a}\right\Vert _{L^{2}(\Omega )}}$} &
        {$\frac{\left\Vert u^{h}-u^{h}\right\Vert _{L^{2}(\Omega )}}{\left\Vert u^{h}\right\Vert _{L^{2}(\Omega )}}$} & {$\frac{\left\Vert a^{h}-I_{h}\bar{a}\right\Vert _{L^{\infty }(\Omega )}}{\left\Vert I_{h}\bar{a}\right\Vert _{L^{\infty }(\Omega )}}$} & 
        {$\frac{\left\Vert u^{h}-u^{h}\right\Vert _{L^{\infty }(\Omega )}}{\left\Vert u^{h}\right\Vert _{L^{\infty }(\Omega )}}$} \\ 
        \midrule
        1e-01 & 9.22e-03 & 9.01e-02 & 3.53e-02 & 5.69e-02 \\ 
        1e-02 & 3.37e-03 & 9.03e-03 & 9.61e-03 & 6.87e-03 \\ 
        1e-03 & 3.23e-03 & 1.09e-03 & 8.70e-03 & 2.28e-03 \\ 
        \bottomrule
    \end{tabular}
\end{table}

\clearpage

\begin{figure}[p]
    \centering
    \begin{subfigure}[t]{0.329\textwidth}
        \centering
        \includegraphics[height=3cm,width=\textwidth]{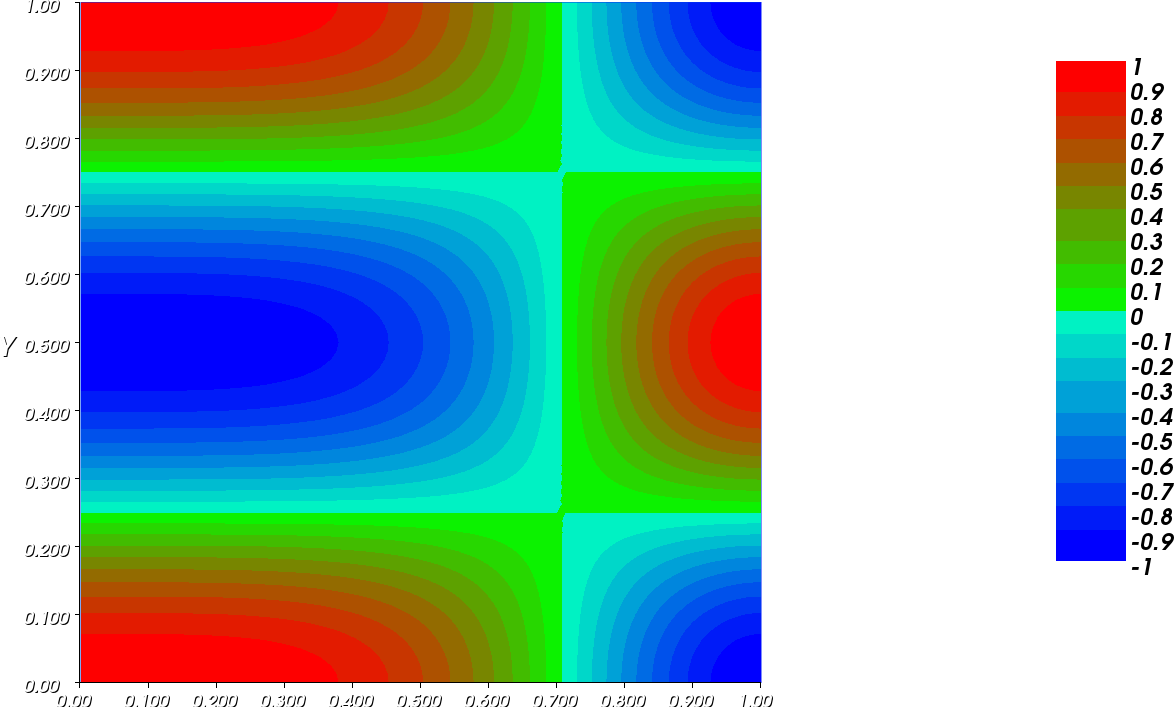}
        \caption{data $z$}
    \end{subfigure}
    \hfill
    \begin{subfigure}[t]{0.329\textwidth}
        \includegraphics[height=3cm,width=\textwidth]{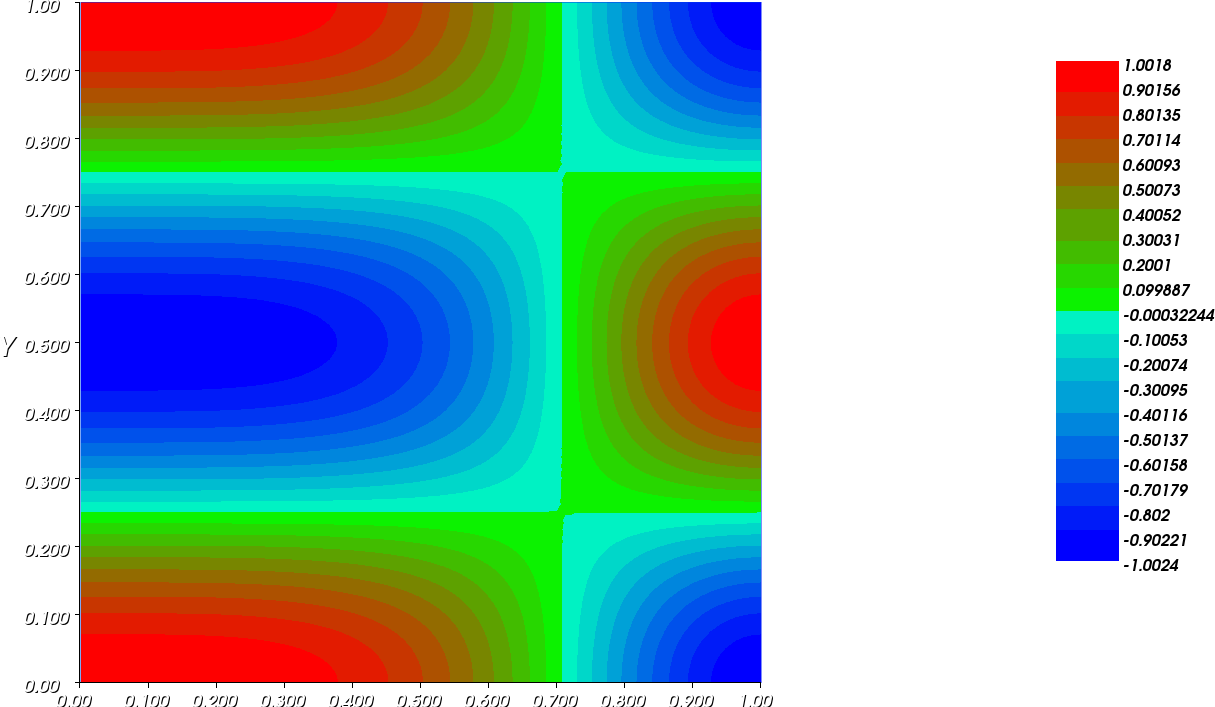}
        \caption{estimated $u$}
    \end{subfigure}
    \hfill
    \begin{subfigure}[t]{0.329\textwidth}
        \includegraphics[height=3cm,width=\textwidth]{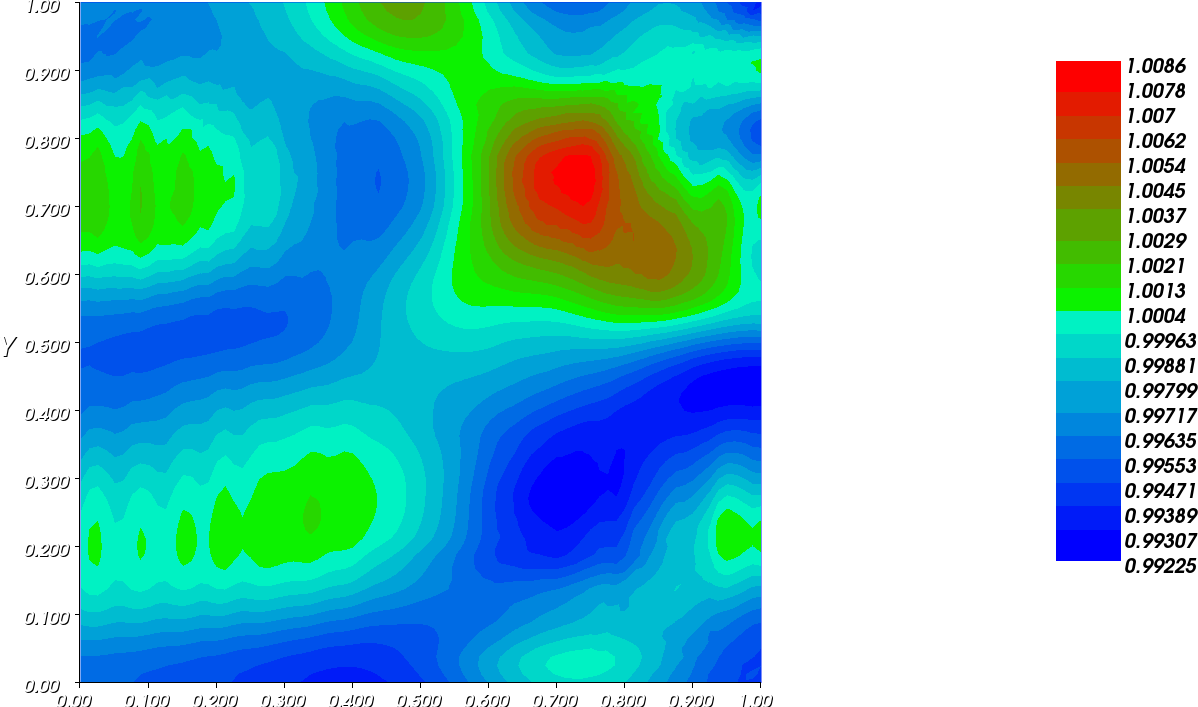}
        \caption{estimated $a$}
    \end{subfigure}
    \caption{Reconstruction with no noise for $h=0.0176777$, $\kappa=\varepsilon=0.0001$ by OLS approach.}
    \label{Fig:OLSFreeNoisekappa0001eps0001}
\end{figure}
\begin{figure}[p]
    \centering
    \begin{subfigure}[t]{0.329\textwidth}
        \centering
        \includegraphics[height=3cm,width=\textwidth]{OLS_Neumann_z_2D_n80}
        \caption{data $z$}
    \end{subfigure}
    \hfill
    \begin{subfigure}[t]{0.329\textwidth}
        \includegraphics[height=3cm,width=\textwidth]{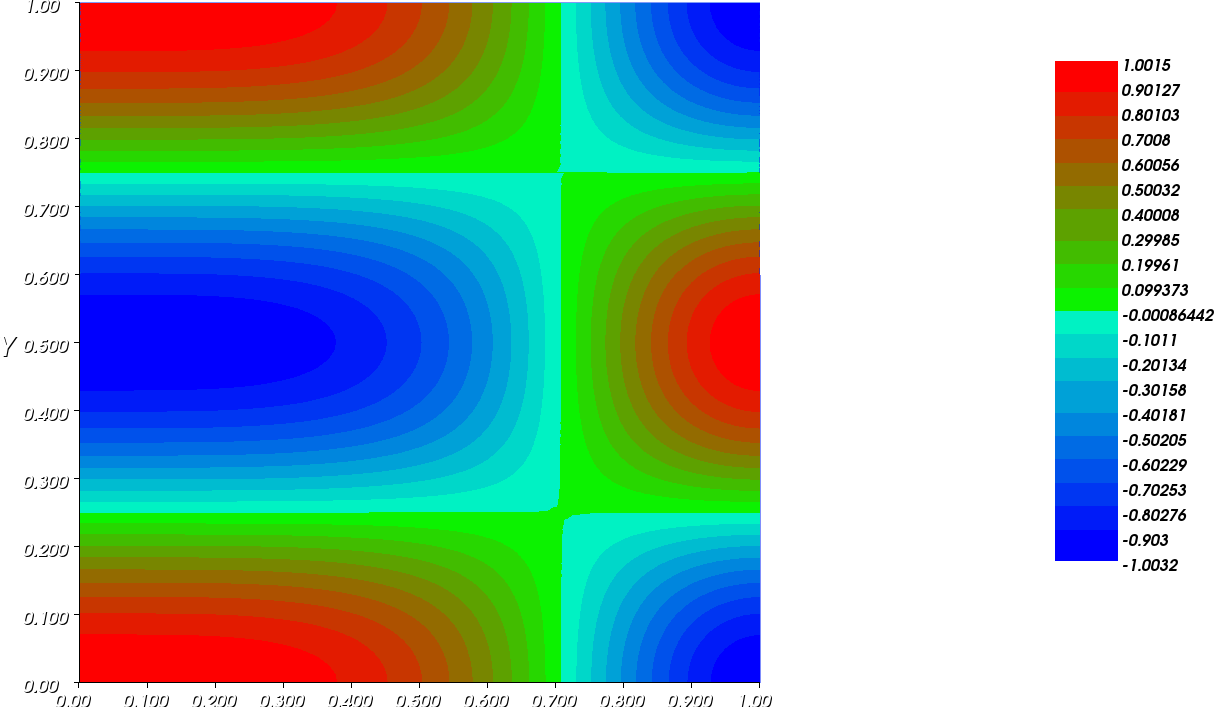}
        \caption{estimated $u$}
    \end{subfigure}
    \hfill
    \begin{subfigure}[t]{0.329\textwidth}
        \includegraphics[height=3cm,width=\textwidth]{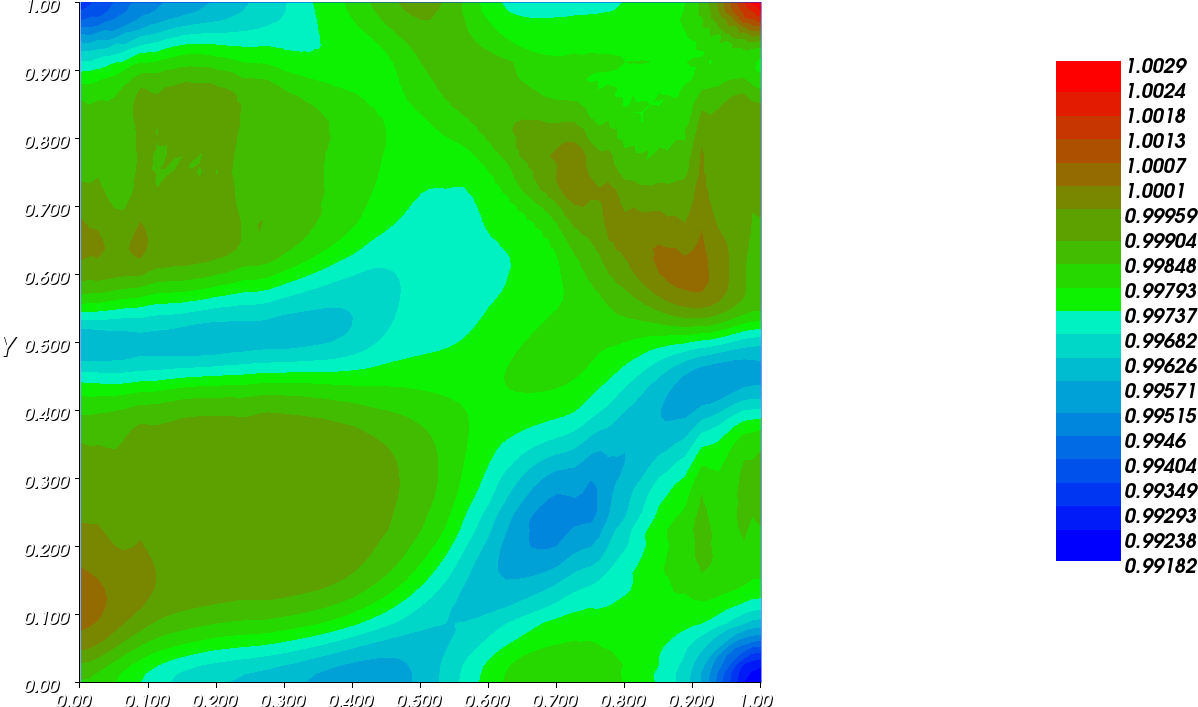}
        \caption{estimated $a$}
    \end{subfigure}
    \caption{Reconstruction with no noise for $h=0.0176777$, $\kappa=0.01$,  $\varepsilon=0.0001$ by MOLS approach.}
    \label{Fig:MOLSFreeNoisekappa0001eps0001}
\end{figure}
\begin{figure}[p]
    \centering
    \begin{subfigure}[t]{0.329\textwidth}
        \centering
        \includegraphics[height=3cm,width=\textwidth]{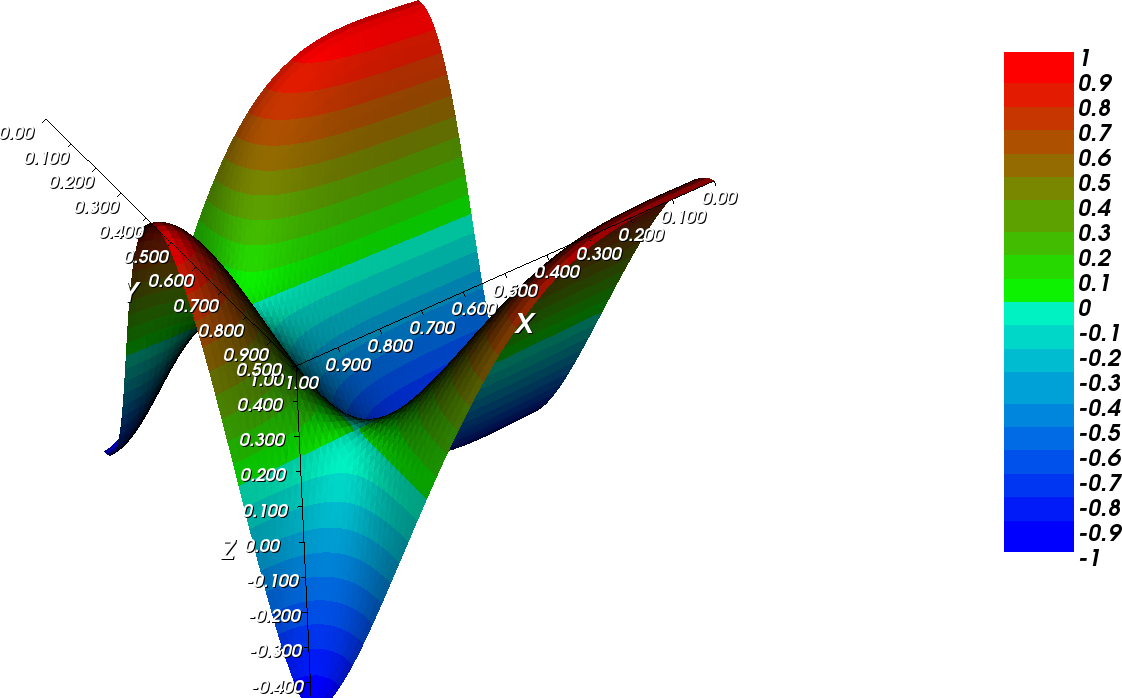}
        \caption{data $z$}
    \end{subfigure}
    \hfill
    \begin{subfigure}[t]{0.329\textwidth}
        \includegraphics[height=3cm,width=\textwidth]{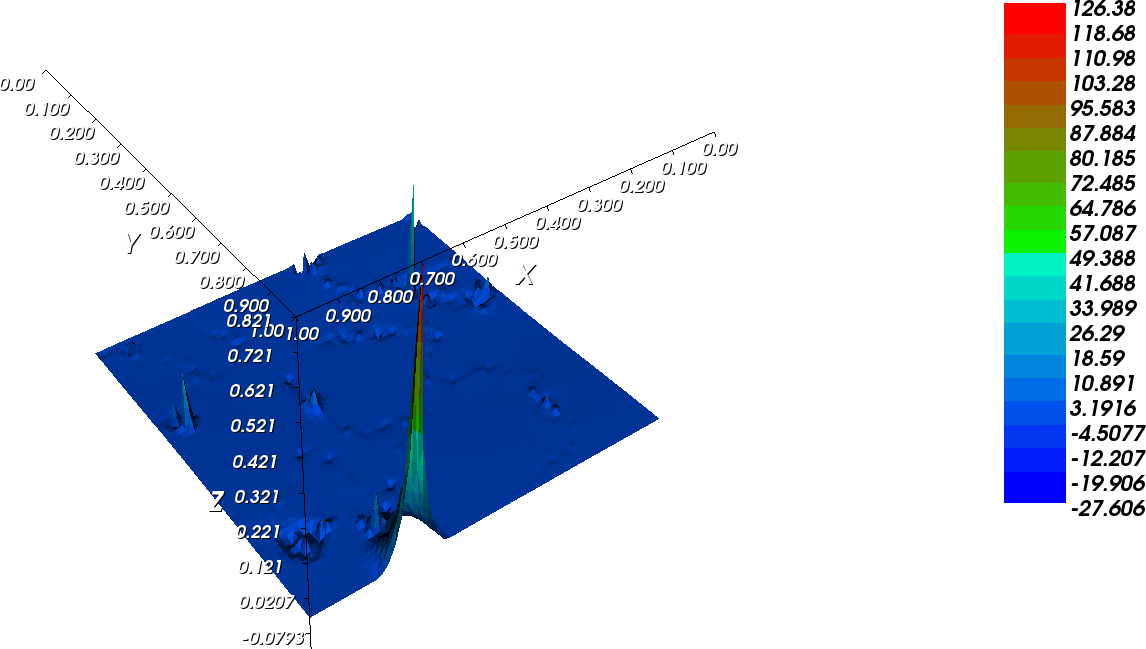}
        \caption{estimated $u$}
    \end{subfigure}
    \hfill
    \begin{subfigure}[t]{0.329\textwidth}
        \includegraphics[height=3cm,width=\textwidth]{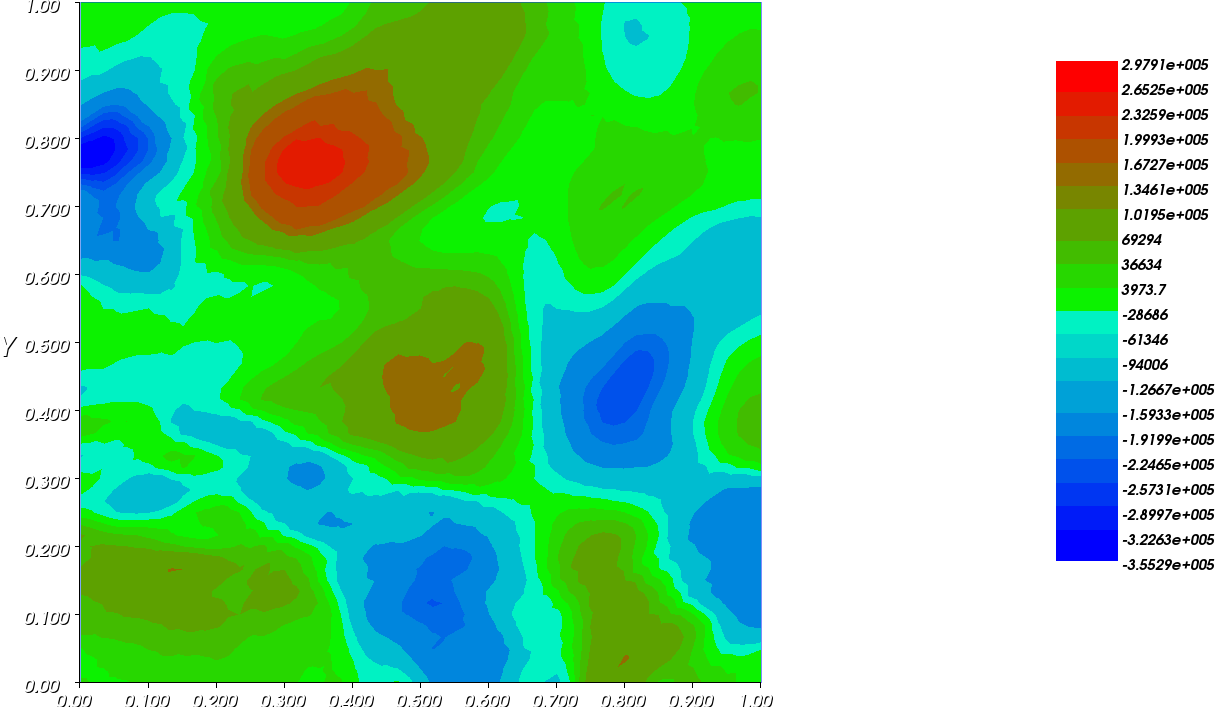}
        \caption{estimated $a$}
    \end{subfigure}
    \caption{Failed reconstruction for $h=0.0235702$, $\kappa=0.0001$,$\varepsilon=0$ by the OLS approach.}
    \label{Fig:FreeNoisekappa0001eps0}
\end{figure}
\begin{figure}[p]
    \centering
    \begin{subfigure}[t]{0.329\textwidth}
        \centering
        \includegraphics[height=3cm,width=\textwidth]{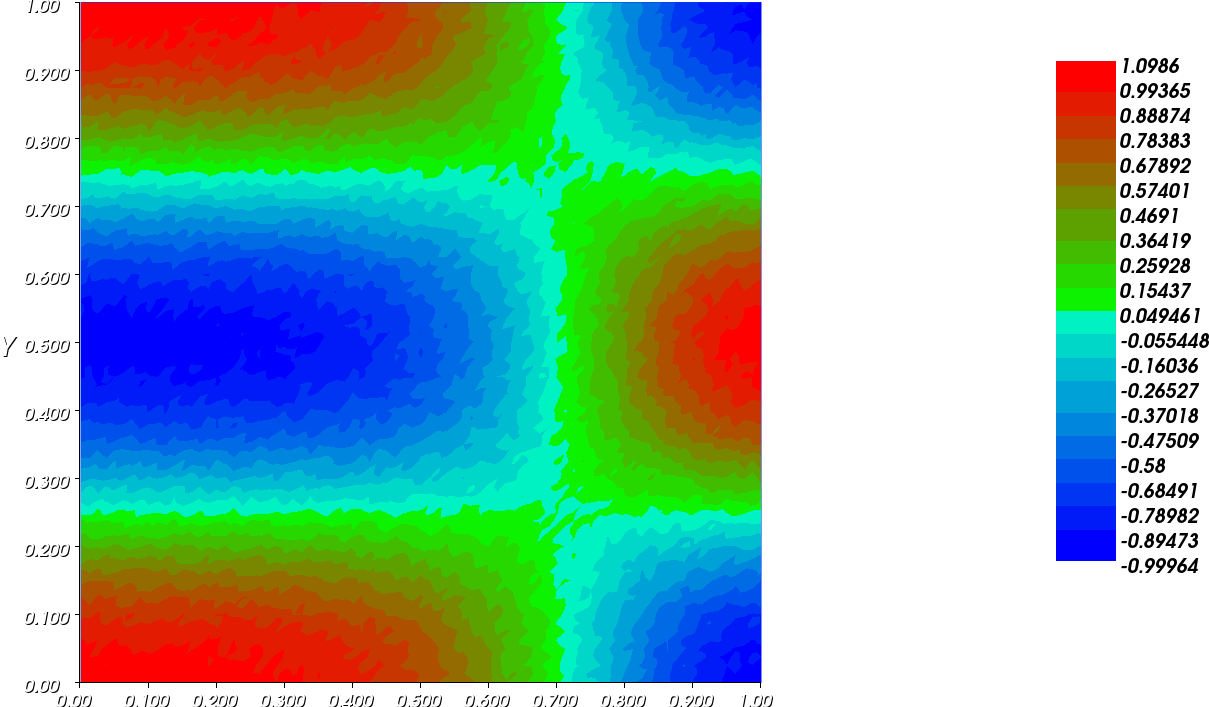}
        \caption{data $z$}
    \end{subfigure}
    \hfill
    \begin{subfigure}[t]{0.329\textwidth}
        \includegraphics[height=3cm,width=\textwidth]{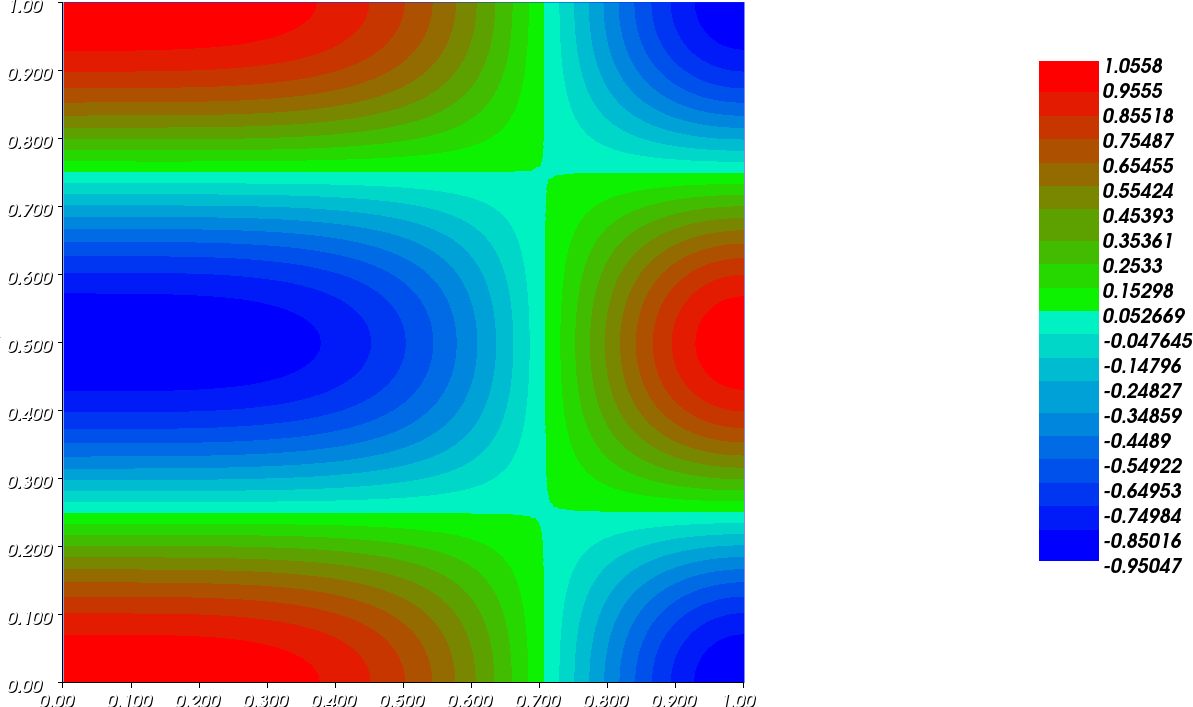}
        \caption{estimated $u$}
    \end{subfigure}
    \hfill
    \begin{subfigure}[t]{0.329\textwidth}
        \includegraphics[height=3cm,width=\textwidth]{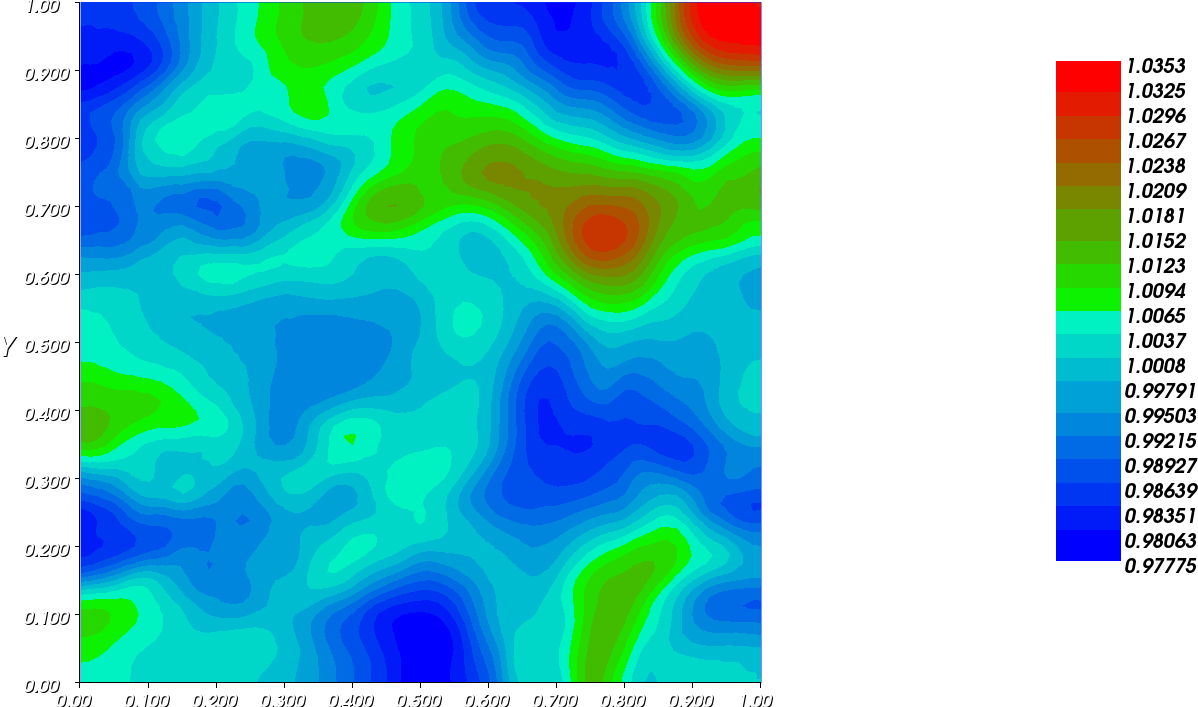}
        \caption{estimated $a$}
    \end{subfigure}
    \caption{Reconstruction  for $h=0.0176777$, $\kappa=\varepsilon=0.0001$, and noise level $\delta_{n}=0.1$ by the OLS approach.}
    \label{Fig:kappa0001eps0001delta01}
\end{figure}

\clearpage

\section{{Concluding Remarks}}

We explored the inverse problem of parameter identification in non-elliptic variational problems by posing optimization problems using the OLS and the MOLS functionals. We regularized the underlying non-elliptic variational problem and studied the features of the regularized parameter-to-solution map. For the set-valued parameter-to-solution map, we relied on the notion of the first-order and the second-order contingent derivatives. To the best  of our knowledge, this is the first work where tools from set-valued optimization have been employed to assist the study of inverse problems of parameter identification. It would be of interest to explore what derivatives of set-valued maps are most convenient
for this kind of research. Detailed numerical experimentation, taking into account the data perturbation, is of paramount importance and will be done in future work. An extension of the present approach to inverse problems in noncoercive variational inequalities also seems to be a promising topic to explore.

\section*{Acknowledgements} 
We are grateful to the reviewers for the careful reading and suggestions. The research of Christian Clason is supported by DFG grant Cl 487/1-1. The research of Akhtar Khan is supported by National Science Foundation grant 005613-002. Miguel Sama's work is partially supported by Ministerio de Econom\`{i}a y Competitividad (Spain), project MTM2015-68103-P and grant  2018-MAT14 (ETSI Industriales, UNED).

%TODO: dois
\printbibliography

\end{document}